\documentclass{siamart250211}
\usepackage[margin=1in]{geometry}
\usepackage[utf8]{inputenc}
\usepackage{csquotes}
\usepackage{amsfonts}
\usepackage{amsmath,amssymb,accents,empheq} 
\usepackage{enumitem}
\usepackage{subcaption}
\usepackage{graphicx}
\usepackage{float}
\usepackage{xcolor}
\usepackage{hyperref}
\usepackage{cite}

\usepackage{cleveref}
\usepackage{ulem}

\usepackage{algorithm}
\usepackage{algpseudocode}
\numberwithin{equation}{section}

\hypersetup{
    colorlinks=true,
    linkcolor=green,
    citecolor=green,
    filecolor=green,      
    urlcolor=green,
}

\newtheorem{remark}{Remark}[section]

\newcommand\SG[1]{\textcolor{black}{#1}}

\usepackage{listings}
 
\definecolor{codegreen}{rgb}{0,0.6,0}
\definecolor{codegray}{rgb}{0.5,0.5,0.5}
\definecolor{codepurple}{rgb}{0.58,0,0.82}
\definecolor{backcolour}{rgb}{0.95,0.95,0.92}

\title{A Sampling-Based Adaptive Rank Approach to the Wigner-Poisson System}
\author{Andrew Christlieb 
\thanks{Department of Computational Mathematics, Science and Engineering, Michigan State University, East Lansing, MI, 48824 (\email{christli@msu.edu})}
\and 
Sining Gong \thanks{Corresponding author. Department of Computational Mathematics, Science and Engineering, Michigan State University, East Lansing, MI, 48824 (\email{gongsini@msu.edu})}
\and 
Jing-Mei Qiu \thanks{Department of Mathematical Sciences, University of Delaware, Newark, DE, 19716 (\email{jingqiu@udel.edu}, \email{nyzheng@udel.edu})} \and 
Nanyi Zheng \footnotemark[3]}

\headers{Adaptive Rank Approach for Wigner-Poisson}{A. Christlieb, S. Gong, J. Qiu, and N. Zheng}

\begin{document}

\maketitle
\begin{abstract}
We develop a mass-conserving, adaptive-rank solver for the 1D1V Wigner–Poisson system. Our work is motivated by applications to the study of the stopping power of $\alpha$ particles at the National Ignition Facility (NIF). In this regime, electrons are in a warm dense state, requiring more than a standard kinetic model. They are hot enough to neglect Pauli exclusion, yet quantum enough to require accounting for uncertainty. The Wigner–Poisson system captures these effects but presents challenges due to its nonlocal nature. Based on a second-order Strang splitting method, we first design a full-rank solver with a structure-preserving Fourier update that ensures the intermediate solutions remain real-valued (up to machine precision), improving upon previous methods. Simulations demonstrate that the solutions exhibit a low rank structure for moderate to high dimensionless Planck constants ($H \ge 0.1$). This observed low rank structure motivates the development of an adaptive-rank solver, built on a Semi-Lagrangian adaptive-rank (SLAR) scheme for advection and an adaptive-rank, structure-preserving Fourier update for the Wigner integral terms, with a rigorous proof of structure-preserving property provided. Our solver achieves $O(N)$ complexity in both storage and computation time, while preserving mass and maintaining momentum accuracy up to the truncation error. The adaptive rank simulations are visually indistinguishable from the full-rank simulations in capturing solution structures. These results highlight the potential of adaptive rank methods for high-dimensional Wigner–Poisson simulations, paving the way toward fully kinetic studies of stopping power in warm dense plasmas.
\end{abstract}


\section{Introduction}

Motivated by our goal to eventually study stopping power of $\alpha$ particles on the National Ignition Facility (NIF), we consider the 
development of an adaptive rank solver for the Wigner-Poisson system \cite{wigner1932quantum,bonitz2016quantum} in one spatial and one velocity dimension (1D1V).  $\alpha$ particles result from the fusion of two hydrogen atoms into an energetic helium particle known as an alpha particle.  It is critical to understand how this particle redistributes its energy to the electrons and other hydrogen atoms in the fusion capsule, as this can impact yield. A complicating factor is that the electrons are in a warm dense quantum state, which has been identified through a combination of experiment and theory \cite{zylstra2022burning,bonitz2016quantum,graziani2014kinetic}.  However, the electrons are hot enough that one can neglect the Pauli exclusion principle, and a mean field approximation to the collective behavior of the quantum electrons is well characterized by the Wigner-Poisson system \cite{graziani2014kinetic}.  A difficulty associated with the Wigner-Poisson system is that it is a nonlocal structure; traditional methods for parallel discretization become intractable due to the high communication cost one needs at each time step. 

 Mesh-based numerical methods for the Wigner–Poisson system have been studied since the mid-1980s. \SG{The methods can be categorized into two classes: monothetic methods and mixed methods, with the latter discussed in the following paragraph.} Frensley's work in \cite{frensley1987wigner} approximated the double integral in the Wigner term using truncated integral sums and employed a first-order upwind scheme for the advection term. Later, in 1989, Jensen and Buot improved upon this by adopting a second-order central difference scheme for advection to address the asymmetric distribution function along the velocity axis in non-biased resonant tunneling devices (RTDs) \cite{jensen1989numerical}. They further extended the scheme to biased RTDs using an upwind-based difference scheme, and enhanced the method by incorporating self-consistency through Poisson's equation and modeling collisions using a BGK collision operator \cite{jensen1991methodology}. Rather than using truncated integral sums, in 1990, Ringhofer introduced the difference-spectral method to handle complex boundary conditions \cite{ringhofer1990spectral}, and later extended his work to self-consistent potentials \cite{ringhofer1992spectral}.  More recently, Chen et al. \cite{chen2019numerical} addressed Wigner systems with unbounded potentials using pseudo-differential operators. In addition, Jiang et al. \cite{jiang2022hybrid} introduced a mixed discretization framework for the Wigner equation, and Sun et al. \cite{sun2024hybrid} developed a high-order summation-by-parts method to handle the double integral in the Wigner term. 
\SG{In 2016, Xiong et al. \cite{xiong2016advective} proposed a spectral method that inverts the advection operator via a resolvent expansion applied to the Wigner term, which is treated explicitly, resulting in a Courant–Friedrichs–Lewy (CFL) restriction.}

Due to the complexity of the Wigner term which involves a nonlocal and highly oscillatory pseudo-differential operator, a variety of splitting methods have been proposed to simplify the system \cite{suh1991numerical,arnold1995operator,arnold1996operator,furtmaier2016semi,dorda2015weno,chen2019high}. 
Suh et al.\cite{suh1991numerical} applied Cheng and Knorr’s scheme for Vlasov–Poisson \cite{cheng1976integration} to the Wigner–Poisson system in 1991.   They were the first to both apply splitting as well as use the model to study quantum plasmas, since the Vlasov–Poisson equation arises as the classical limit of the Wigner–Poisson system as quantum effects become negligible.
Arnold and Ringhofer proposed first- and second-order splitting methods with stability analysis in \cite{arnold1995operator,arnold1996operator}.  The analysis in \cite{arnold1996operator} covered the case of Suh's method, which they show is unconditionally stable.  Later, Chen et al. introduced a high-order splitting method in a 2022 paper \cite{chen2022higher}. These methods decompose the system into subproblems, allowing the application of different appropriate schemes to each component. In particular, techniques from computational fluid dynamics can be used for the advection equation without the Wigner term, while the Wigner term can be handled separately. For example,
in \cite{suh1991numerical}, the method makes use of a second order Semi-Lagrangian (SL) for the advection term  and uses an analytical spectral update on the Wigner term to avoid having a CFL restriction. In \cite{suh1991numerical}, avoiding the CFL comes at the cost of conservation.  In \cite{dorda2015weno}, the advection term is treated with a fifth-order finite volume Weighted Essentially Non-Oscillatory (WENO) scheme using explicit time stepping, and the Wigner term is treated with the finite element method. In \cite{chen2019high}, the WENO scheme is used for spatial advection, and the spectral method is used for the Wigner term. 
We note that while we focus on mesh-based methods, there has been substantial work on Monte Carlo methods for Wigner equations \cite{sellier2014benchmark,sellier2015introduction,nedjalkov2004unified,shao2015comparison}. These methods, based on particle models, are applied to single-particle or two-particle Wigner equations instead of the full Wigner-Poisson system and it  remains unclear whether they can be extended to kinetic descriptions. A signed-particle Monte Carlo method was proposed in \cite{nedjalkov2004unified} for the simulation of RTDs. Later, in 2014, the Wigner Monte Carlo method was extended to simulate a Gaussian wavepacket interacting with a potential barrier, and compared with results from \cite{shao2011adaptive}, which used the cell-average spectral element method to analyze numerical accuracy.

In this paper, we first improve the full-rank solver for the Wigner–Poisson system originally proposed in \cite{suh1991numerical} and analyzed in \cite{arnold1996operator}, which employs a mesh-based method combined with second-order Strang splitting. Specifically, we incorporate a structure-preserving modification to the Fourier update step, ensuring that the intermediate pre-Inverse Fast Fourier Transform (IFFT) solution's conjugate-symmetric structure property is preserved, which is not satisfied in the previous solver. For the physical space advection, we employ the semi-Lagrangian (SL) method with high-order WENO reconstruction from \cite{qiu2010conservative}.  By performing singular value decomposition (SVD) of the full-rank solution, we observe that the solutions of the Wigner–Poisson system are intrinsically low rank, especially when quantum effects are significant such as the case $H = 1,8$, where $H$ denotes the dimensionless Planck constant defined in \eqref{eqn:quantum parameter}. This observation motivates the development of an adaptive rank approach with operator splitting for the Wigner–Poisson system. The low rank aspect of this work is a natural extension of our recent Semi-Lagrangian Adaptive Rank (SLAR) method for the Vlasov–Poisson system \cite{zheng2025semi}. The SL formulation enables large time step sizes beyond the CFL restriction. Following the work of \cite{suh1991numerical, cheng1976integration}, we perform operator splitting: we use the SLAR method \cite{zheng2025semi} with the high order WENO reconstruction developed in \cite{qiu2010conservative} for the physical space; then we extend the adaptive rank approach to the Fourier space to address the complications arising from the nonlocal phase space operator. The adaptive rank approach to the Fourier space requires additional attention due to its conjugate-symmetric structure property which is preserved by applying an improved adaptive rank approach. A rigorous proof is provided for such a structure-preserving adaptive rank approach. By fully exploiting the low rank structure of the Wigner-Poisson system, the proposed SLAR–Fourier solver achieves $\mathcal{O}(N)$ complexity, where $N$ denotes the number of mesh points per dimension. Such complexity greatly reduces memory demands and has the potential to 
significantly speed up the full rank mesh-based 
solver.  We note that this work  paves the road to affordable high dimensional Wigner-Poisson simulations using a Tensor CUR extension to the algorithm presented here. To the best of our knowledge, this is the first paper with arbitrary CFL to solve the Wigner-Poisson system in an adaptive rank setting with operator splitting. 
In addition, we propose to conserve mass by applying a Lagrange multiplier scheme to the solution obtained by our adaptive rank Wigner-Poisson solver after each time step. 

We note that low rank evolution of time-dependent multi-dimensional partial differential equations (PDEs) got started from the pioneering work on dynamic low rank (DLR) methods for matrix differential equations in \cite{koch2007dynamical}. Since then, there have been  rapid research developments in the field, including the vast development of the dynamical low rank (DLR) approach  ~\cite{koch2007dynamical, lubich2014projector, einkemmer2018low, ceruti2022unconventional, dektor2021dynamic} and the step-and-truncate (SAT) approach ~\cite{kormann2015semi, dektor2021rank} with applications to various high dimensional PDEs. Among existing efforts, the low rank approach for the Vlasov-Poisson system is of particular relevance to this work. \cite{einkemmer2018low} developed the DLR and \cite{GuoVlasovFlowMap2022} proposed the SAT low rank approaches to the Vlasov-Poisson model; in \cite{einkemmer2021mass, guo2024local, guo2024conservative} conservative projections are proposed to exactly preserve the macroscopic observables  even with low rank truncation. 

Our paper is laid out in the following way: in Section 2, we introduce the non-dimensionalized Wigner-Poisson system and present numerical studies of the full rank solution to motivate the development of adaptive rank solvers; in Section 3, we describe our improved full rank solver based on operator splitting with the structure-preserving property and our proposed adaptive rank solver with the adaptation of the adaptive cross approximation to the solution matrix in physical and phase space; in Section 4, we demonstrate the efficacy of the proposed method and verify its low rank complexity; and finally, Section 5 concludes the paper.


\section{The Wigner–Poisson model and its solutions' rank structure}

This section begins by reviewing the Wigner–Poisson model and presenting its non-dimensional form. Motivated by conditions relevant to the National Ignition Facility (NIF), where the dimensionless Planck constant \(H\) (a scaled measure of quantum effects, defined in \Cref{subsec:Wigner poisson system}) is typically of order one, we then conduct a numerical study in one spatial and one velocity dimension (1D1V) setting using a full-rank solver in \Cref{subsec:rank analysis}. The goal is to examine the rank structure of Wigner–Poisson solutions in this moderate-\(H\) regime, based on the singular value decomposition (SVD) of full-rank results, and thereby assess the potential suitability of low-rank methods for approximating the nonlocal Wigner operator.

\subsection{Wigner–Poisson system}\label{subsec:Wigner poisson system}

The Wigner–Poisson system in 1D1V, as presented in~\cite[Chapter~4]{haas2011quantum}, takes the form:
\begin{subequations}\label{eqn:WPmodel}
\begin{align}
\frac{\partial \tilde{f}}{\partial \tilde{t}} + \tilde{v} \frac{\partial \tilde{f}}{\partial s} 
&= -\frac{iem_e}{2\pi \hbar^2} \iint d\tilde{v}'\, ds'\, \exp\left( i m_e \frac{(\tilde{v}' - \tilde{v}) s'}{\hbar} \right) \left[ \phi\left(s + \frac{s'}{2}\right) - \phi\left(s - \frac{s'}{2}\right) \right] \tilde{f}(s, \tilde{v}', \tilde{t}), \\
\frac{\partial^2 \phi}{\partial s^2} &= -\frac{e}{\epsilon_0} \left( \int d\tilde{v}\, \tilde{f} - n_0 \right),
\end{align}
\end{subequations}
where \(\tilde{f}(s, \tilde{v}, \tilde{t})\) is the Wigner distribution function in phase space, and \(\phi(s)\) denotes the electrostatic potential. The physical constants are: \(m_e\) the particle mass, \(e\) the elementary charge, \(\epsilon_0\) the vacuum permittivity, \(n_0\) the background number density, \(\hbar\) the reduced Planck constant, and \(\tilde{t}\) the physical time. 

We now nondimensionalize the Wigner–Poisson system to facilitate scaling analysis and numerical discretization. Let \(\tau\), \(l\), and \(\bar{\phi}\) denote the characteristic time, length, and electrostatic potential scales, respectively. The corresponding nondimensional variables are defined as:
\[
\tilde{t} = \tau t,\quad s = l x,\quad \phi = \bar{\phi} \Phi,\quad \tilde{v} = \frac{l}{\tau} v,\quad f = \frac{l}{n_0 \tau} \tilde{f}.
\]
Substituting these into the Wigner–Poisson system~\eqref{eqn:WPmodel} yields:
\begin{subequations}
\begin{align}
\frac{\partial f}{\partial t} + v \frac{\partial f}{\partial x} &= -\frac{iC}{2\pi H^2} \iint dv' dx' \, \exp\left(i \frac{v' - v}{H} x'\right) \left[\Phi\left(x + \frac{x'}{2}\right) - \Phi\left(x - \frac{x'}{2}\right)\right] f(x, v', t), \\
\frac{\partial^2 \Phi}{\partial x^2} &= -D \left(\int f \, dv - \rho_0 \right), \qquad \text{with } \rho_0 = 1. \label{subeq:non-dim poisson}
\end{align}
\end{subequations}
The resulting dimensionless parameters are given by:
\[
C = \frac{e \bar{\phi} \tau^2}{m_e l^2},\quad H = \frac{\tau \hbar}{m_e l^2},\quad D = \frac{e n_0 l^2}{\bar{\phi} \epsilon_0}.
\]

To further simplify the system and highlight key physical scalings, we choose the potential scale as
\[
\bar{\phi} = \frac{e n_0 l^2}{\epsilon_0},
\]
and define the time and length scales based on characteristic plasma parameters: time is scaled by the plasma frequency \(\omega_{pe}\), and space by the Debye length \(\lambda_D\), given respectively by
\[
\omega_{pe} = \sqrt{ \frac{e^2 n_0}{m_e \epsilon_0} }, \qquad \lambda_D = \sqrt{ \frac{ \epsilon_0 k_B T_0}{n_0 e^2} },
\]
where \(T_0\) denotes the reference temperature. 
Under these scalings, the velocity scale becomes the thermal velocity,
\[
\frac{l}{\tau} = \lambda_D \omega_{pe} 
= \sqrt{ \frac{ \epsilon_0 k_B T_0}{n_0 e^2} } \cdot \sqrt{ \frac{e^2 n_0}{m_e \epsilon_0} } 
= \sqrt{ \frac{k_B T_0}{m_e} } = v_{th},
\]
which is a natural scale in kinetic descriptions of plasmas. As a result, the dimensionless parameters simplify significantly:
\begin{equation}\label{eqn:quantum parameter}
C = \frac{e \bar{\phi} }{m_e l^2 \omega_{pe}^2} 
= \frac{e}{m_e l^2} \cdot \frac{e n_0 l^2}{\epsilon_0} \cdot \frac{m_e \epsilon_0}{e^2 n_0} = 1, \quad
H = \frac{ \hbar}{m_e \lambda_D^2 \omega_{pe}}, \quad
D = \frac{e n_0 l^2}{\bar{\phi} \epsilon_0} = 1.
\end{equation}
Here, \(H\) emerges as the sole remaining parameter, quantifying the strength of quantum effects in the dimensionless formulation.

Thus, the non-dimensional Wigner–Poisson system takes the form:
\begin{subequations}\label{eqn:nonD-WP-system}
\begin{align}
\frac{\partial f}{\partial t} + v \frac{\partial f}{\partial x} &= -\frac{i}{2 \pi H^2} \iint dv' dx' \, \exp\left( i \frac{v' - v}{H} x' \right) \left[ \Phi\left(x + \frac{x'}{2}\right) - \Phi\left(x - \frac{x'}{2}\right)\right] f(x,v', t), \label{eq:wigner} \\
- \frac{\partial^2 \Phi }{\partial x^2} &= \int f \, dv - 1. \label{eq:poisson}
\end{align}
\end{subequations}
Note that the dimensionless quantum parameter \(H\), defined in~\eqref{eqn:quantum parameter}, remains a free input that must be specified numerically. Unlike the Vlasov–Poisson system, the Wigner–Poisson model requires both the Debye length and plasma frequency to be prescribed in order to determine the appropriate value of \(H\). As \(H \to 0\), the system~\eqref{eqn:nonD-WP-system} formally converges to the classical Vlasov–Poisson limit; see Appendix~\ref{appendixA} for a detailed derivation.

\subsection{Low Rank Structure via the Lens of Full Rank Solutions}\label{subsec:rank analysis}

Motivated by preliminary numerical observations, we hypothesize that solutions to the Wigner–Poisson system tend to exhibit low-rank structure at late times when the dimensionless Planck constant \(H\) is moderately large (e.g., \(H \geq 0.1\)). In this regime, quantum effects become significant, but the solution complexity appears to remain compressible in the phase space. This behavior stands in contrast to the classical Vlasov–Poisson system, where the solution typically develops increasingly fine-scale filamentation over time, leading to rapid rank growth unless collisional damping is introduced.

To verify and quantify this phenomenon, we perform a systematic study based on full-rank numerical simulations using the methods described in Section~\ref{sec:Numerical methods}, without employing any low-rank approximations. Specifically, we examine the singular value decay of Wigner–Poisson solutions for several values of \(H = 0.1,\ 0.5,\ 1,\ 8\), and under varying spatial resolutions. While similar trends are observed under different initial conditions and domain sizes—including cases involving strong Landau damping—the results presented here focus on the two-stream instability problem on the domain \(x \in [0, 4\pi]\), \(v \in [-2\pi, 2\pi]\), with the initial condition
\begin{equation}\label{eqn:ini_TSI}
    f_0(x,v) = \frac{v^2}{\sqrt{8\pi}} \left(2 + \cos\left(\frac{x}{2}\right)\right) e^{-v^2/2}.
\end{equation}

We first examine the solution behavior under different mesh resolutions \(N_x = 512,\ 1024,\ 2048\) and values of \(H = 0.1,\ 0.5,\ 1,\ 8\). As a reference benchmark, we identify \(N_x = 4096\) as sufficiently resolved for all cases. Notably, solutions for larger values of \(H\) exhibit convergence on coarser meshes. Using these resolved full-rank solutions, we analyze how the time evolution of the numerical rank, measured by the singular value decomposition (SVD), depends on both the mesh resolution and the quantum parameter \(H\). These findings serve as the motivation for developing a low-rank approximation framework for the Wigner–Poisson system.

For all visualizations, the color bar is fixed to the range \(f(t,x,v) \in [-0.3,\ 0.525]\). The colormap is deliberately split to emphasize quantum tunneling effects: values in the range \([-0.3,\ 0)\) are rendered in grayscale (black to white), while values in \([0,\ 0.525]\) are shown using a blue–yellow–red scale. This design choice highlights a fundamental feature in the Wigner framework: the possibility of quantum tunneling, in which electrons are no longer confined by classical potentials but can propagate through them. Such tunneling regions are indicated by the negativity of the Wigner function, which, unlike the classical Vlasov–Poisson distribution function, does not represent a true probability density.

Figures~\ref{fig:fullrank_T45_1} and~\ref{fig:fullrank_T45_2} present the phase-space solutions of the Wigner–Poisson system at time \(T = 45\) for four representative values of the quantum parameter \(H\). In each figure, the left panel shows the solution computed at a resolution of \(512 \times 512\), while the right panel corresponds to a fully resolved solution. The numerical results agree closely with those reported in~\cite{suh1991numerical}, and will serve as reference solutions when evaluating the accuracy of the proposed low-rank Wigner–Poisson solver. A key observation from these figures is that increasing \(H\) leads to a visible reduction in phase-space complexity: the solutions exhibit less fine-scale filamentation and smoother structures. This behavior differs significantly from that of the Vlasov–Poisson system, where phase-space dynamics typically evolve toward increasingly intricate filaments over time due to the absence of quantum smoothing effects.

The progressive simplification of phase-space structures over time suggests that the solution may exhibit low rank structure. In Figure \ref{fig:fullrank_H_SVDrank}, we plot the SVD rank required to capture $95\%$, $99\%$, $99.99\%$, $99.9999\%$, and $99.999999\%$ of the total energy as a function of time for a solution on a $2048 \times 2048$ mesh. Subfigures (a) through (d) correspond to $H = 0.1$, $0.5$, $1$, and $8$, respectively. For energy levels up to $99.9999\%$, the rank tends to level off and remain bounded in time. Furthermore, the maximum rank required to reach a given energy threshold decreases with increasing $H$. At the highest energy level ($99.999999\%$), corresponding to the full-order solution, the rank grows slowly over time and appears to stabilize more rapidly for larger $H$.

Figure \ref{fig:fullrank_H_relaSVDrank} shows the normalized rank (rank divided by $N_x$) as a function of time for SVD energy thresholds of $99\%$ and $99.9999\%$, across meshes of size $256 \times 256$, $512 \times 512$, $1024 \times 1024$, and $2048 \times 2048$. Two key trends emerge: first, the normalized rank required to achieve a given resolution decreases as \(N_x\) increases; second, for fixed mesh size, the normalized rank required to represent the solution decreases significantly as \(H\) increases.

In summary, this numerical study demonstrates that the Wigner–Poisson system exhibits strong low-rank structure in the phase space, particularly for moderate to large values of the quantum parameter \(H\). 
 
\begin{figure}[htbp]
    \centering
    \begin{subfigure}{0.49\textwidth}
        \centering
        \includegraphics[width=\linewidth]{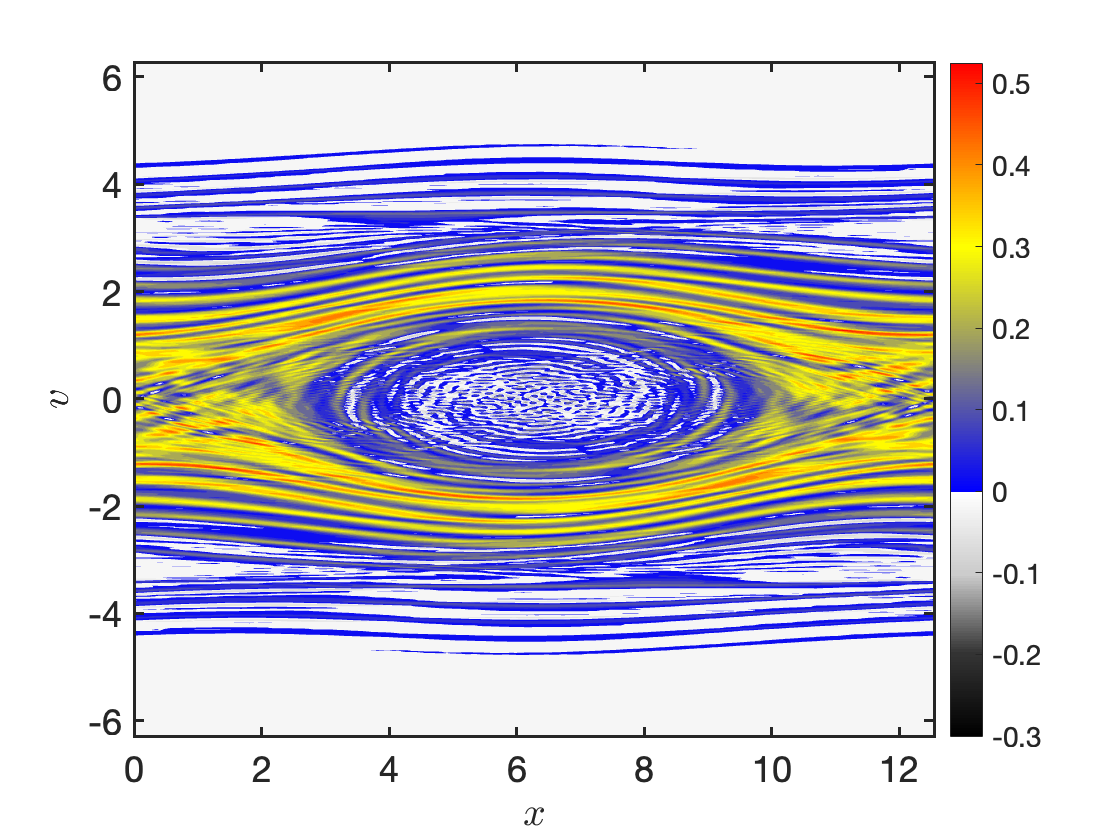}
        \caption{$H = 0.1$, and $N_x = 512$.}
        \label{fig:fullrank_H01_T45_Nx512}
    \end{subfigure}
    \begin{subfigure}{0.49\textwidth}
        \centering
        \includegraphics[width=\linewidth]{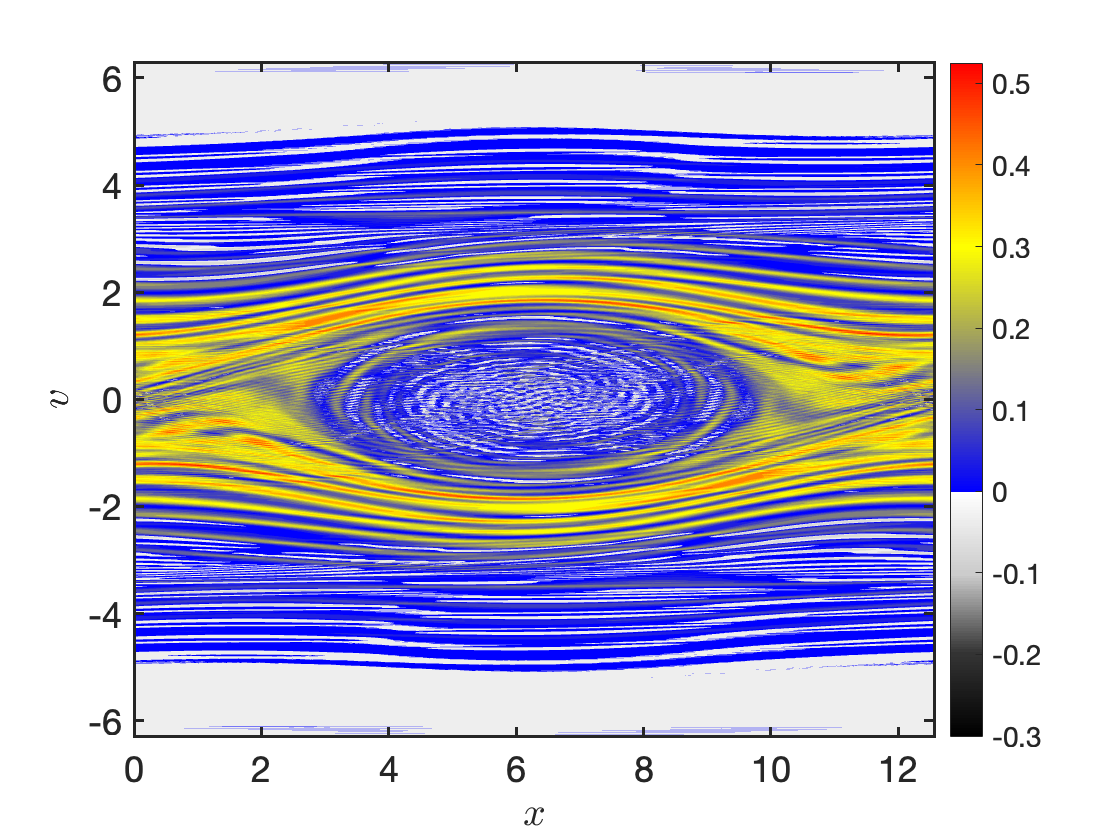}
        \caption{$H = 0.1$, and $N_x =2048$.}
        \label{fig:fullrank_H01_T45_Nx2048}
    \end{subfigure}
    
    \begin{subfigure}{0.49\textwidth}
        \centering
        \includegraphics[width=\linewidth]{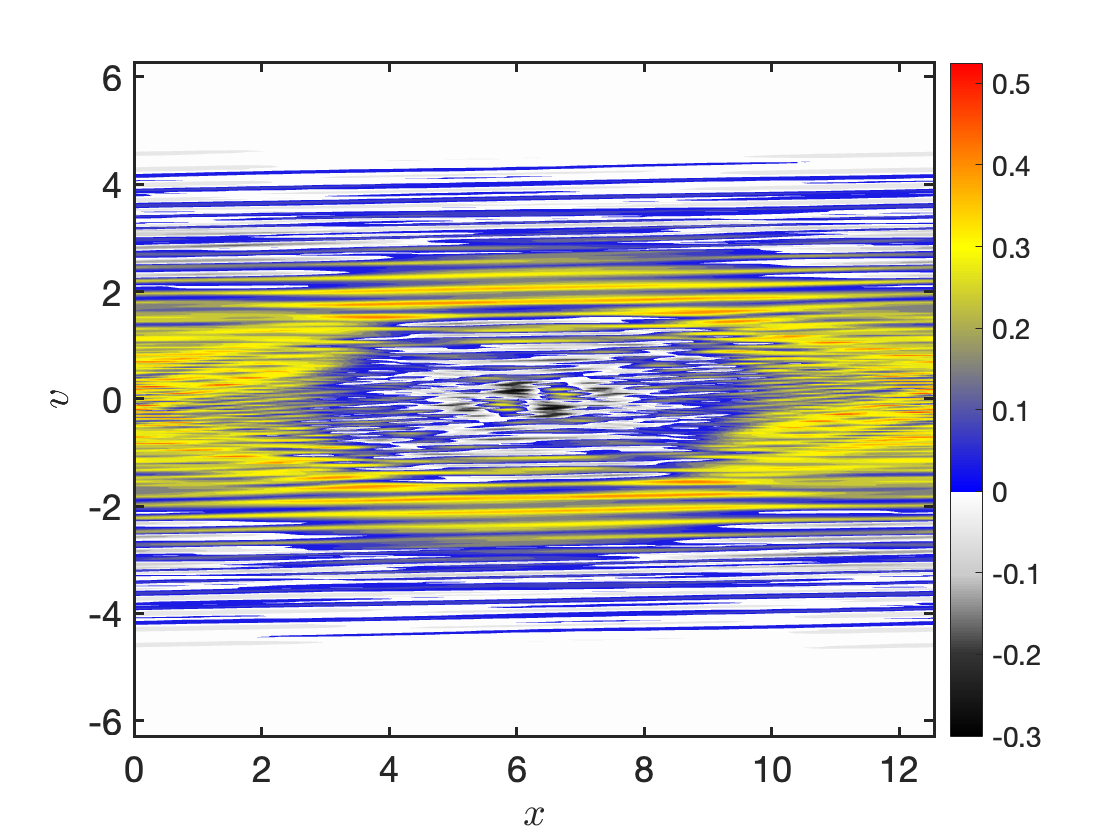}
        \caption{$H = 0.5$, and $N_x = 512$.}
        \label{fig:fullrank_H05_T45_Nx512}
    \end{subfigure}
    \begin{subfigure}{0.49\textwidth}
        \centering
        \includegraphics[width=\linewidth]{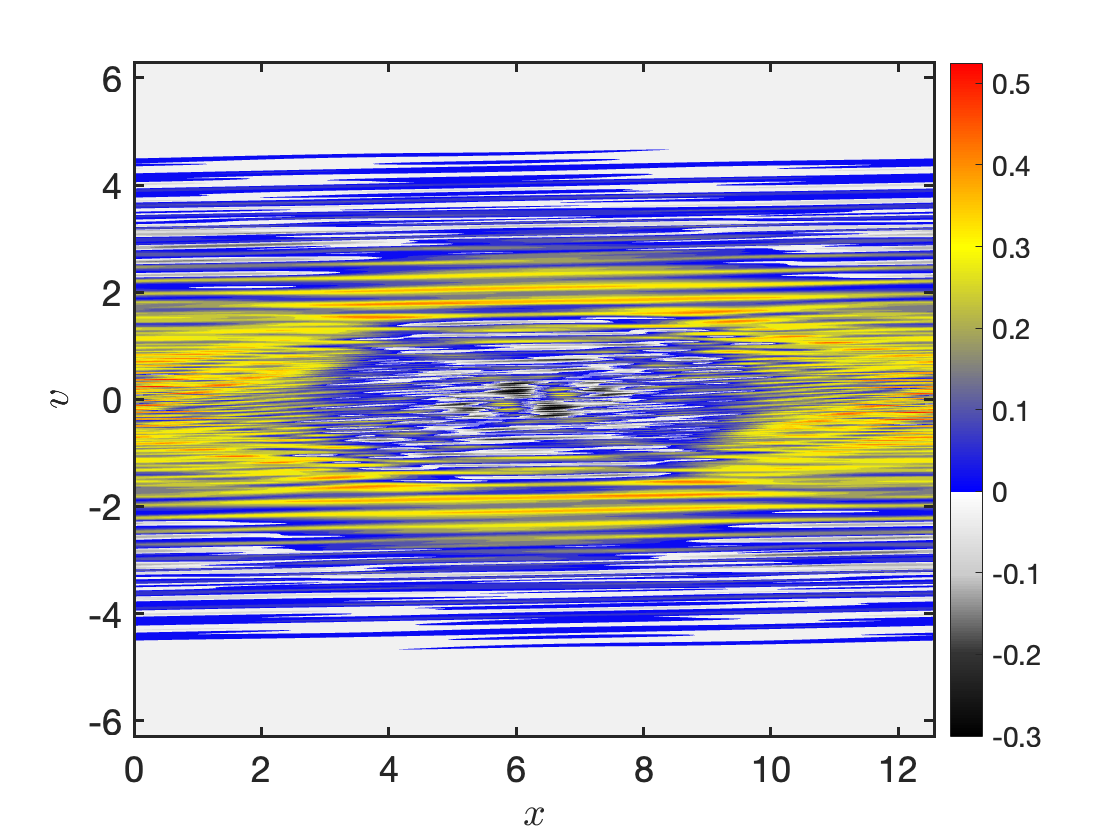}
        \caption{$H = 0.5$, and $N_x =2048$.}
        \label{fig:fullrank_H05_T45_Nx2048}
    \end{subfigure}  


    \caption{Phase-space solutions of the Wigner–Poisson system at time \(T = 45\). This figure shows the phase space solutions from full-rank simulations for different meshes and values of $H$ at $T = 45$. The left column shows results at a resolution of $512 \times 512$, while the right column shows “fully resolved” solutions. Two key observations: unlike in the Vlasov–Poisson system, it is possible to resolve the solution over long time runs on a fixed mesh; and as $H$ increases, fine-scale features in phase space diminish.}
    \label{fig:fullrank_T45_1}
\end{figure}

\begin{figure}[htbp]
    \centering
    \begin{subfigure}{0.49\textwidth}
        \centering
        \includegraphics[width=\linewidth]{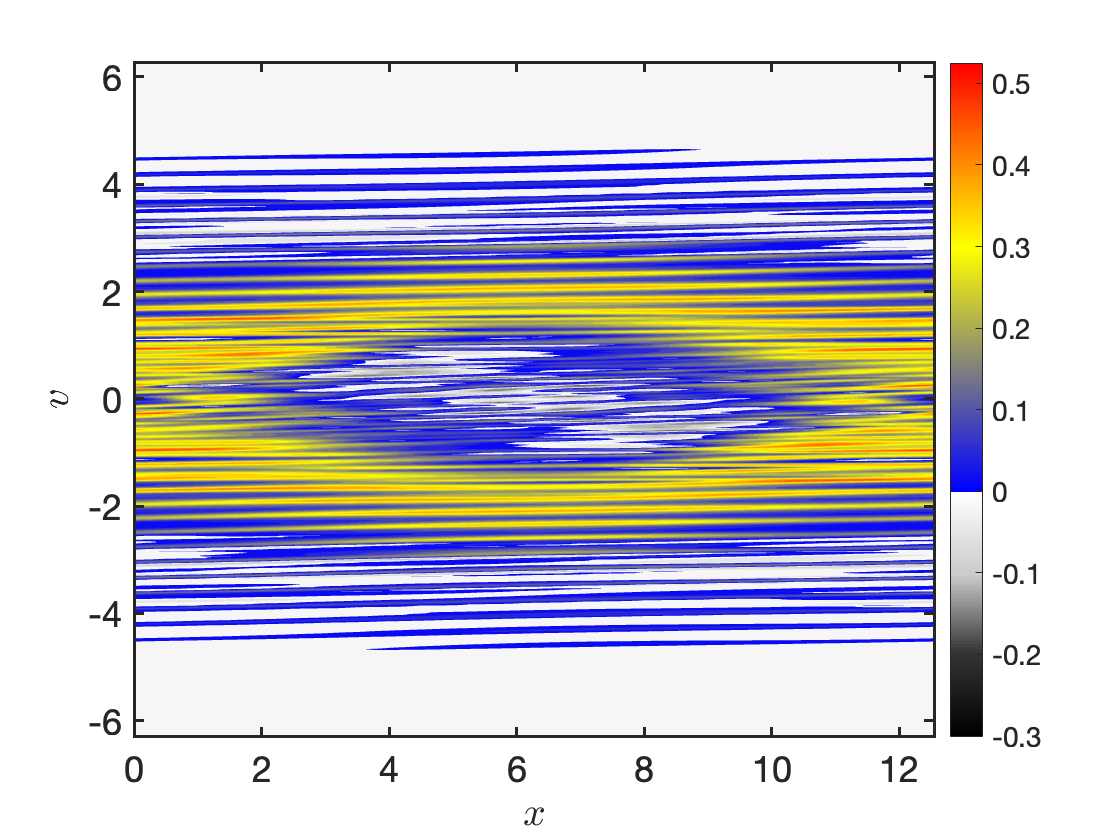}
        \caption{$H = 1$, and $N_x = 512$.}
        \label{fig:fullrank_H1_T45_Nx512}
    \end{subfigure}
    \begin{subfigure}{0.49\textwidth}
        \centering
        \includegraphics[width=\linewidth]{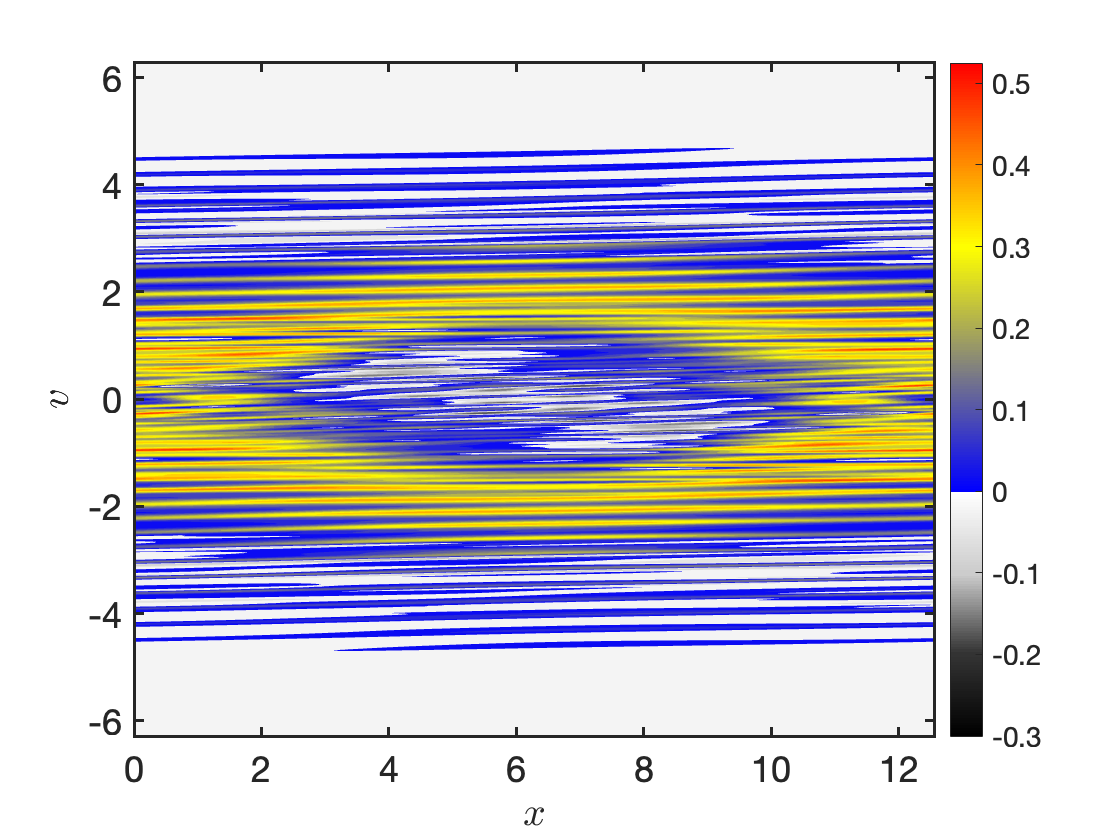}
        \caption{$H = 1$, and $N_x =1024$.}
        \label{fig:fullrank_H1_T45_Nx1024}
    \end{subfigure}
    
    \begin{subfigure}{0.49\textwidth}
        \centering
        \includegraphics[width=\linewidth]{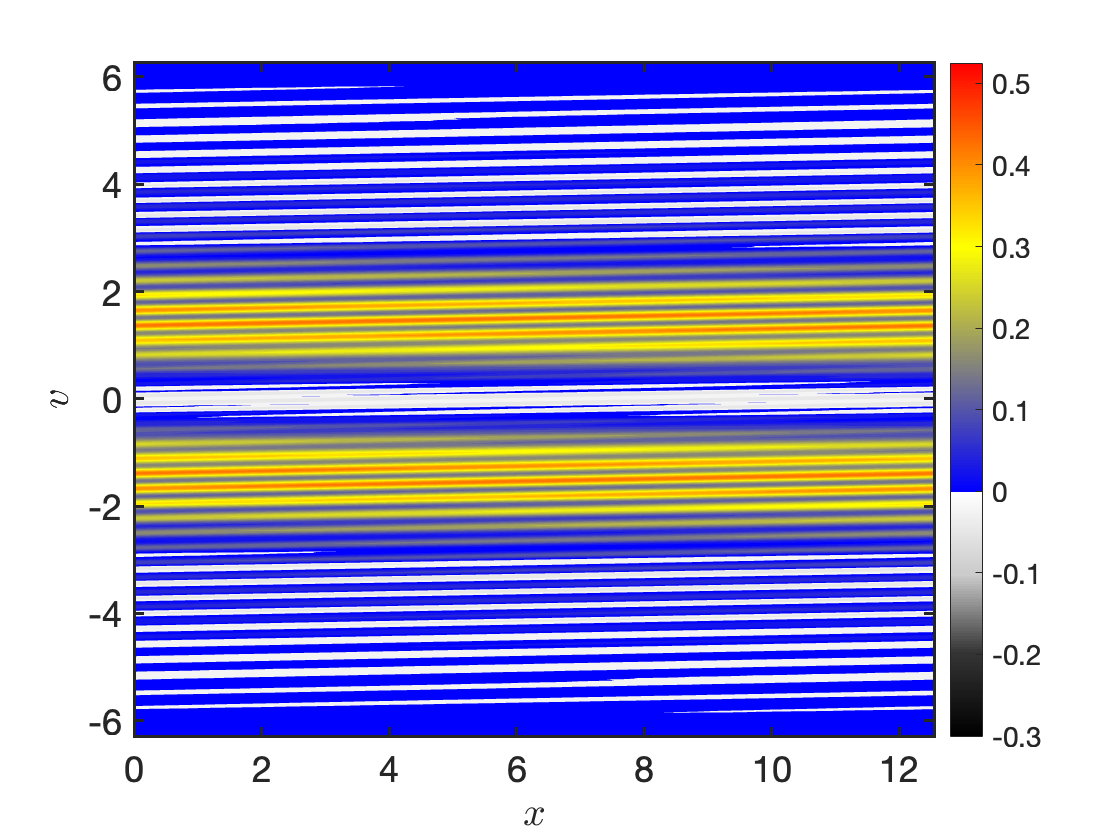}
        \caption{$H = 8$, and $N_x = 512$.}
        \label{fig:fullrank_H8_T45_Nx512}
    \end{subfigure}
    \begin{subfigure}{0.49\textwidth}
        \centering
        \includegraphics[width=\linewidth]{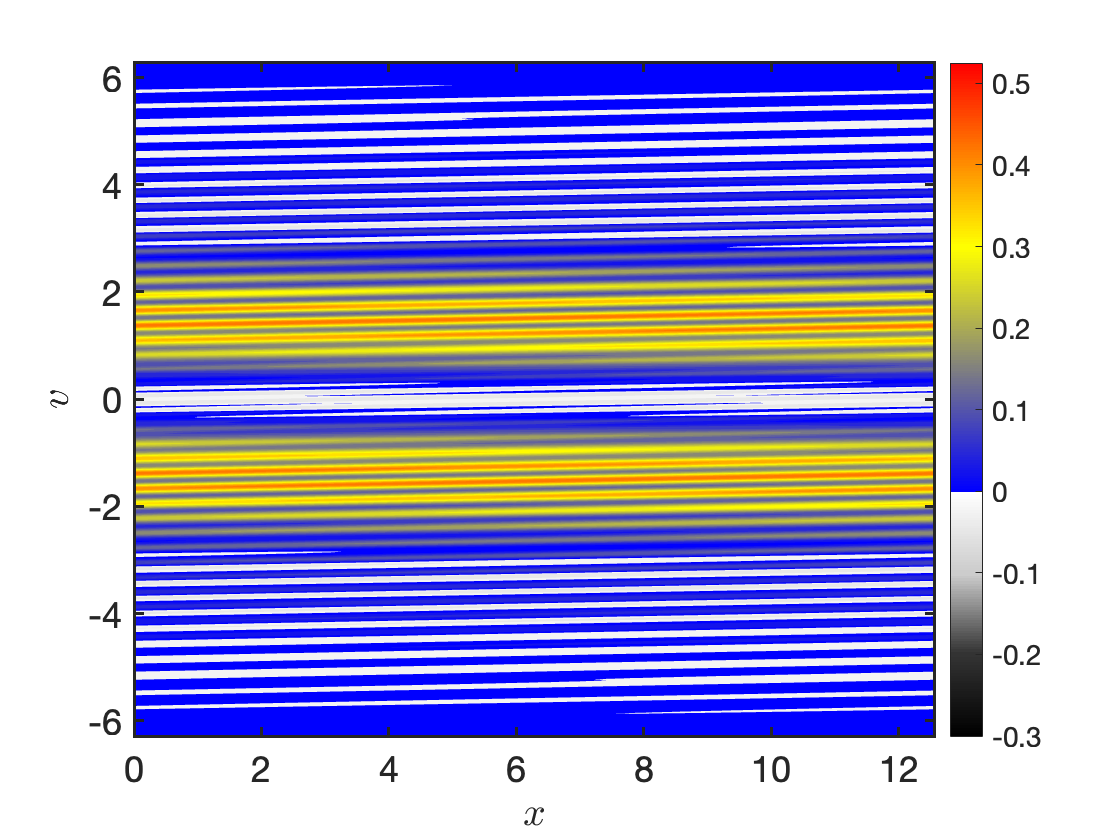}
        \caption{$H = 8$, and $N_x =1024$.}
        \label{fig:fullrank_H8_T45_Nx1024}
    \end{subfigure}
    \caption{Phase-space solutions of the Wigner–Poisson system at time \(T = 45\) (continued from Figure~\ref{fig:fullrank_T45_1}). See the figure caption of Figure~\ref{fig:fullrank_T45_1} for a detailed description. }
    \label{fig:fullrank_T45_2}
\end{figure}

\begin{figure}[htbp]
    \centering
    \begin{subfigure}{0.49\textwidth}
        \centering
        \includegraphics[width=\linewidth]{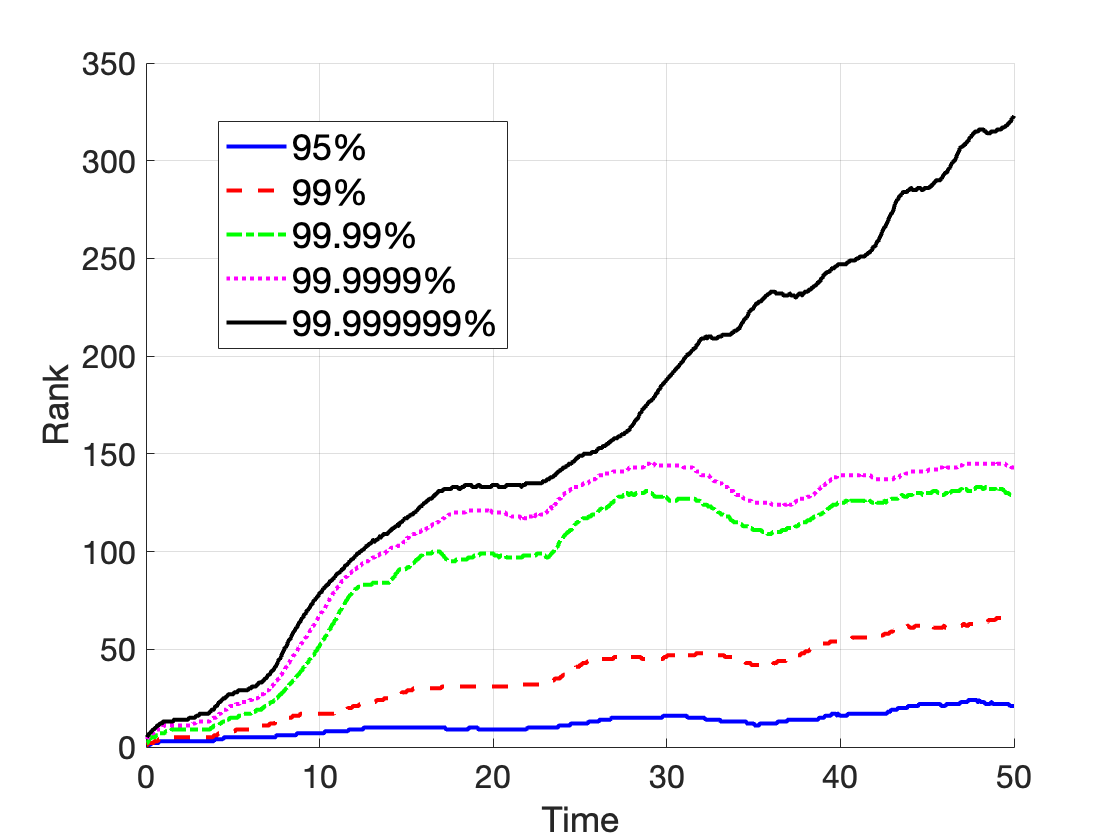}
        \caption{$H = 0.1$, and $N_x = 2048$.}
        \label{fig:fullrank_H01_SVDrank_2048}
    \end{subfigure}
    \begin{subfigure}{0.49\textwidth}
        \centering
        \includegraphics[width=\linewidth]{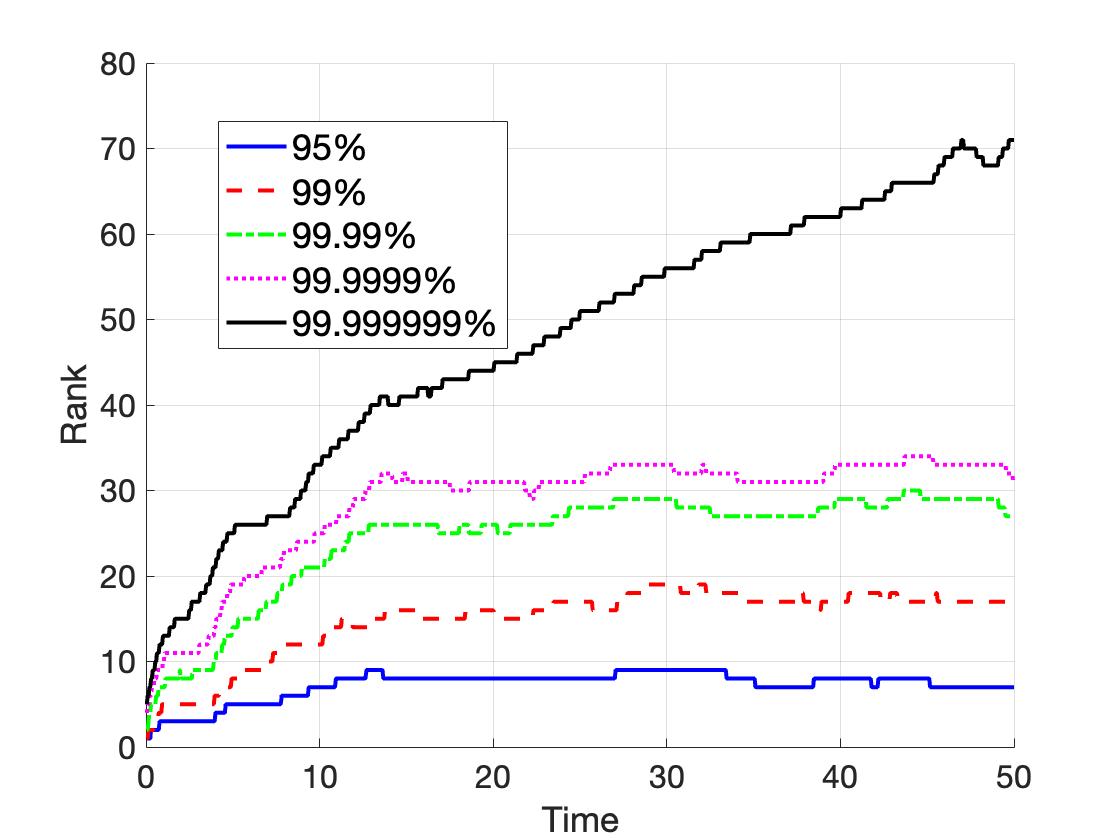}
        \caption{$H = 0.5$, and $N_x =2048$.}
        \label{fig:fullrank_H05_SVDrank_2048}
    \end{subfigure} 
    \begin{subfigure}{0.49\textwidth}
        \centering
        \includegraphics[width=\linewidth]{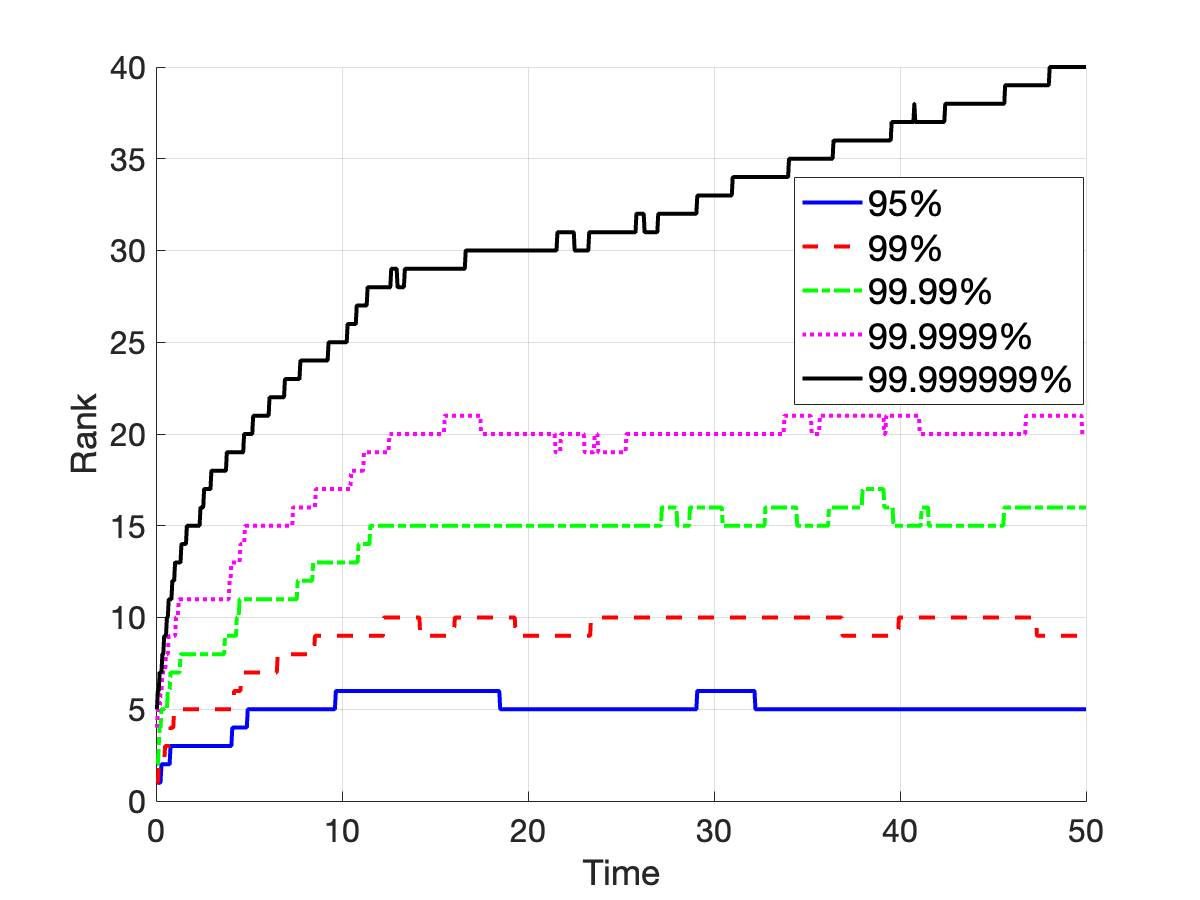}
        \caption{$H = 1$, and $N_x = 2048$.}
        \label{fig:fullrank_H1_SVDrank_2048}
    \end{subfigure}
    \begin{subfigure}{0.49\textwidth}
        \centering
        \includegraphics[width=\linewidth]{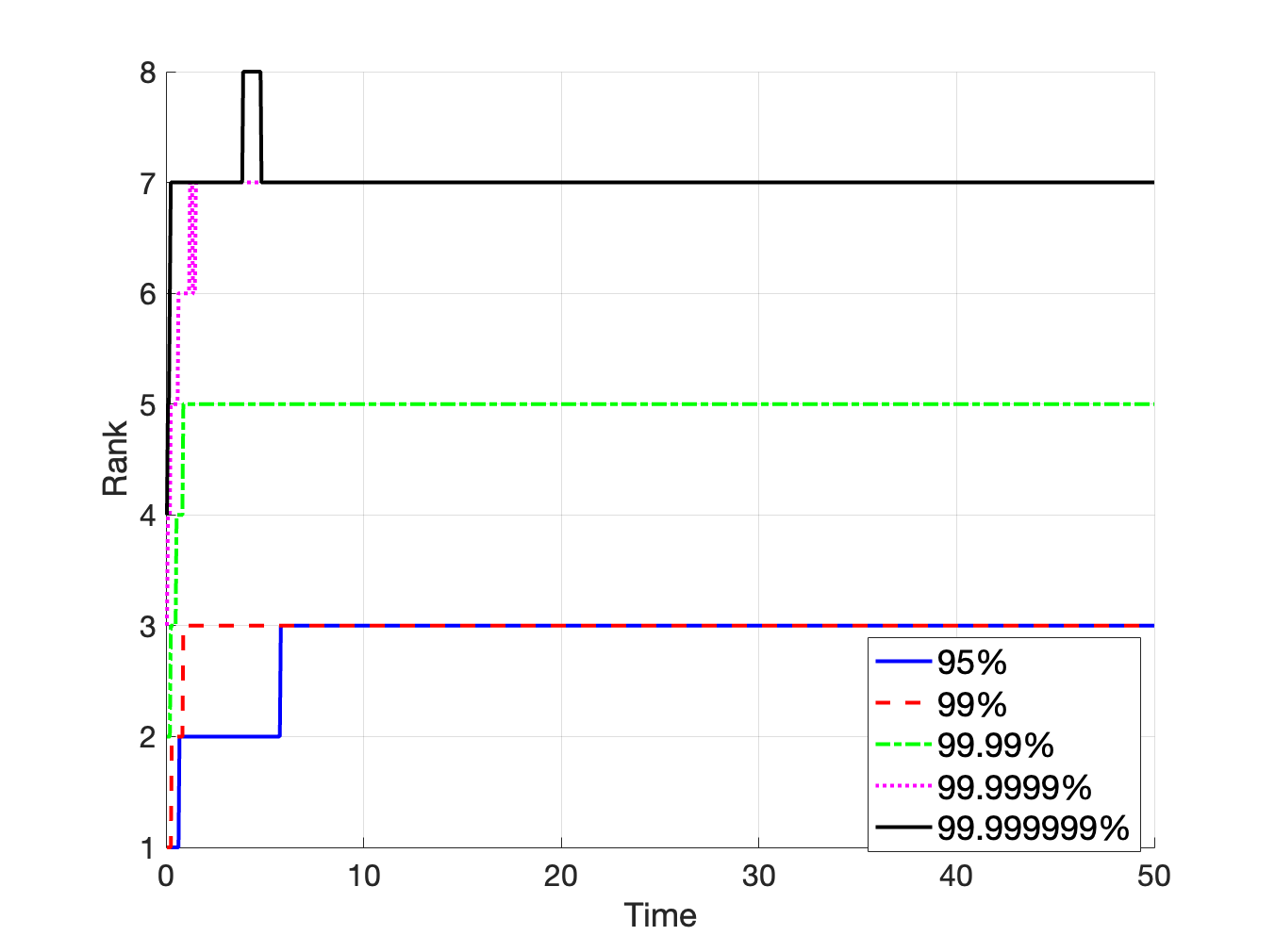}
        \caption{$H = 8$, and $N_x =2048$.}
        \label{fig:fullrank_H8_SVDrank_2048}
    \end{subfigure}  
    \caption{Time evolution of SVD ranks. This figure shows the SVD ranks of full-rank solutions as a function of time for different values of $H \in \{0.1, 0.5, 1, 8\}$ on a $2048 \times 2048$ mesh. The rank is plotted for five energy thresholds: $95\%$, $99\%$, $99.99\%$, $99.9999\%$, and $99.999999\%$. The maximum rank required to capture the solution at a given energy threshold decreases significantly as $H$ increases.}
    \label{fig:fullrank_H_SVDrank}
\end{figure}

\begin{figure}[htbp]
    \centering
    \begin{subfigure}{0.49\textwidth}
        \centering
        \includegraphics[width=\linewidth]{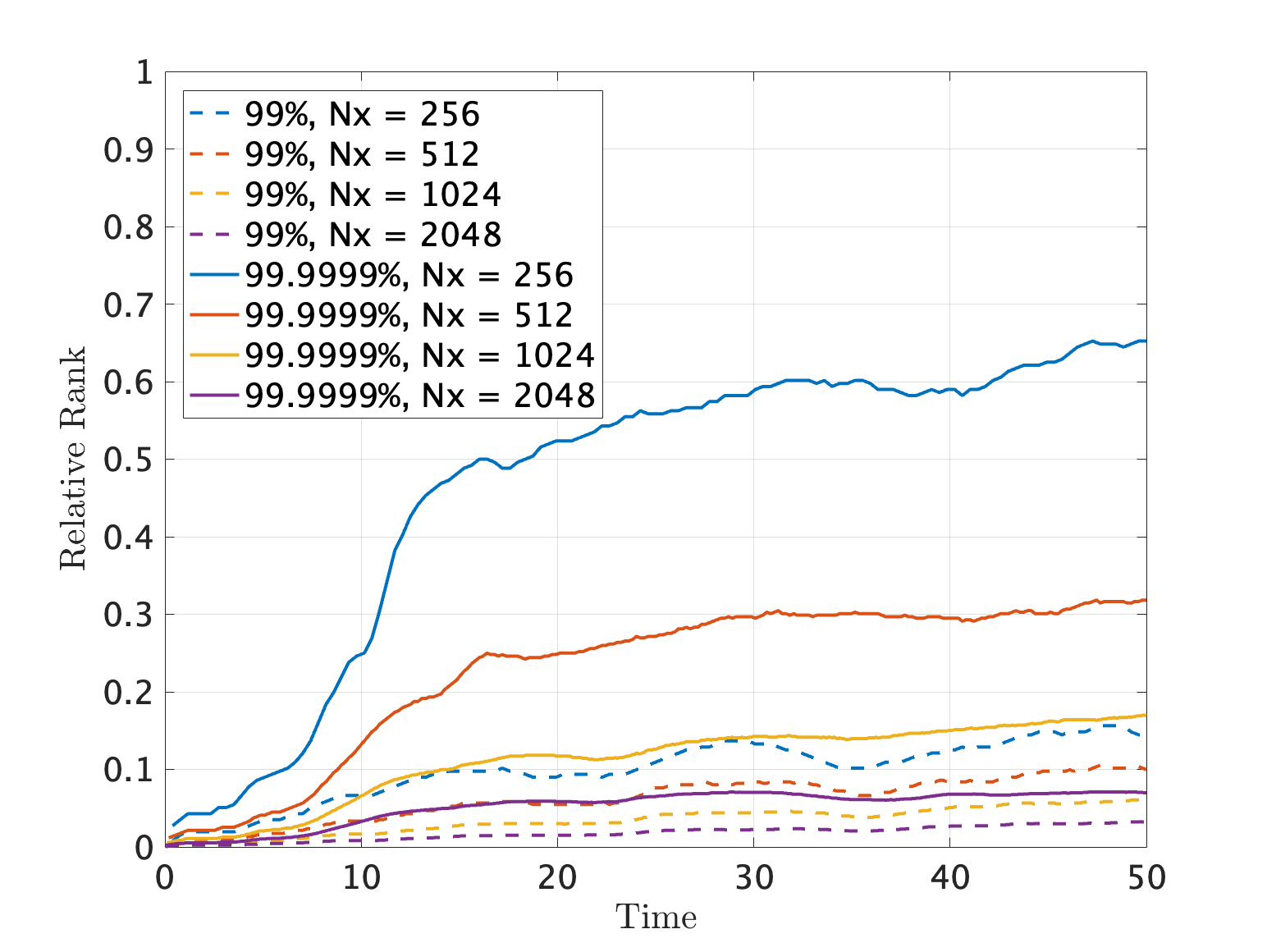}
        \caption{$H = 0.1$.}
        \label{fig:fullrank_H01_relaSVDrank_2048}
    \end{subfigure}
    \begin{subfigure}{0.49\textwidth}
        \centering
        \includegraphics[width=\linewidth]{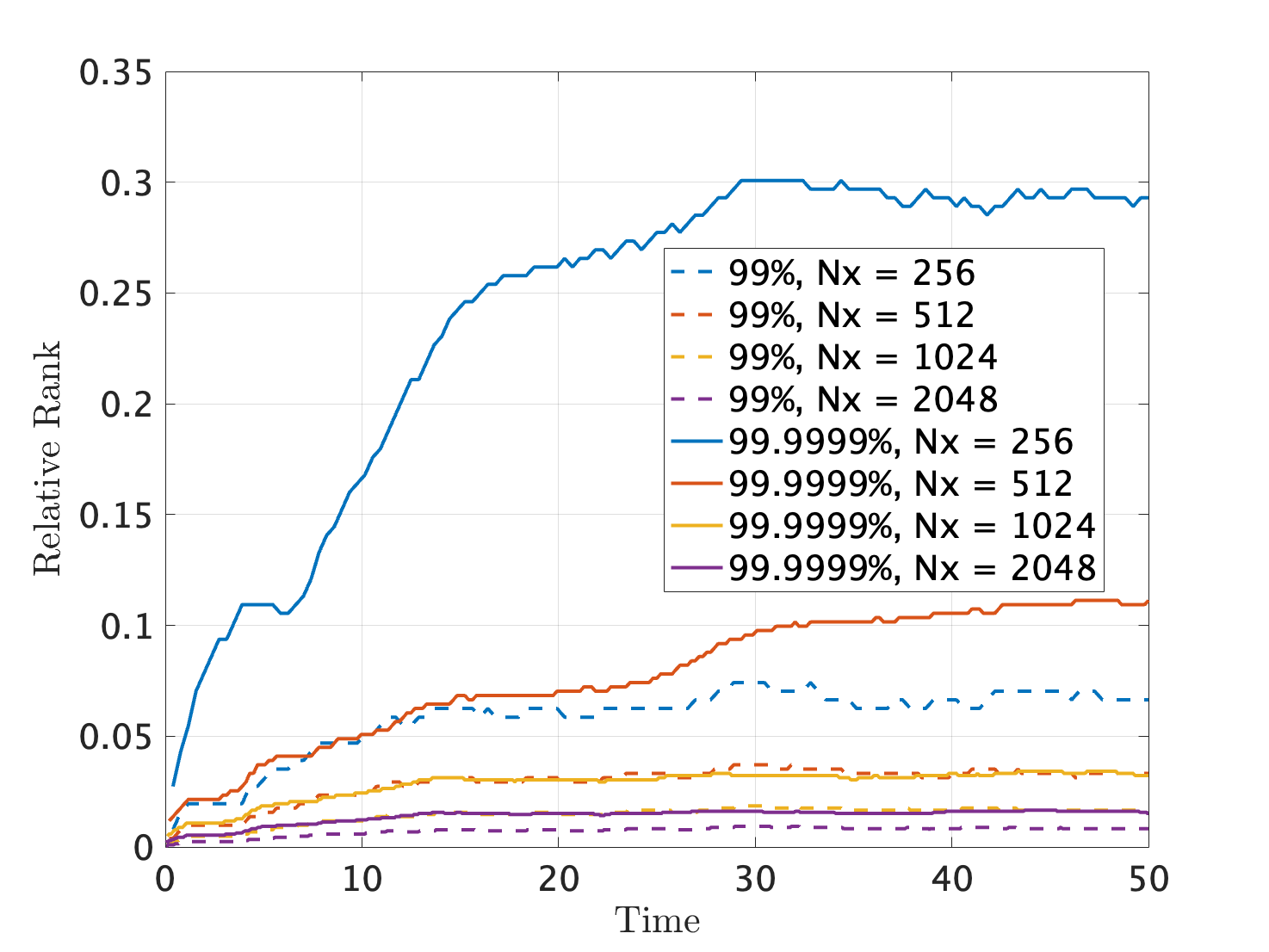}
        \caption{$H = 0.5$.}
        \label{fig:fullrank_H05_relaSVDrank_2048}
    \end{subfigure}  
    \begin{subfigure}{0.49\textwidth}
        \centering
        \includegraphics[width=\linewidth]{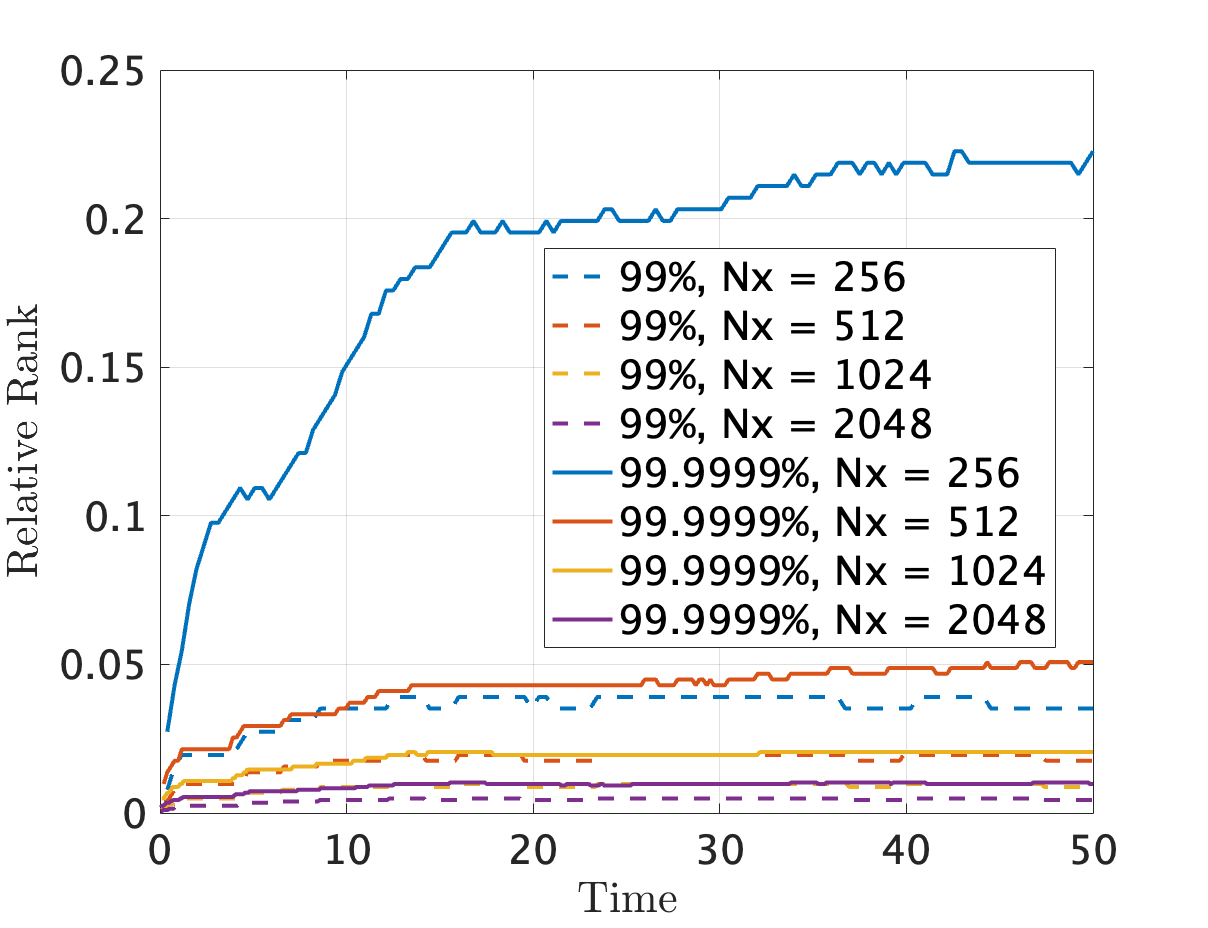}
        \caption{$H = 1$.}
        \label{fig:fullrank_H1_relaSVDrank_2048}
    \end{subfigure}
    \begin{subfigure}{0.49\textwidth}
        \centering
        \includegraphics[width=\linewidth]{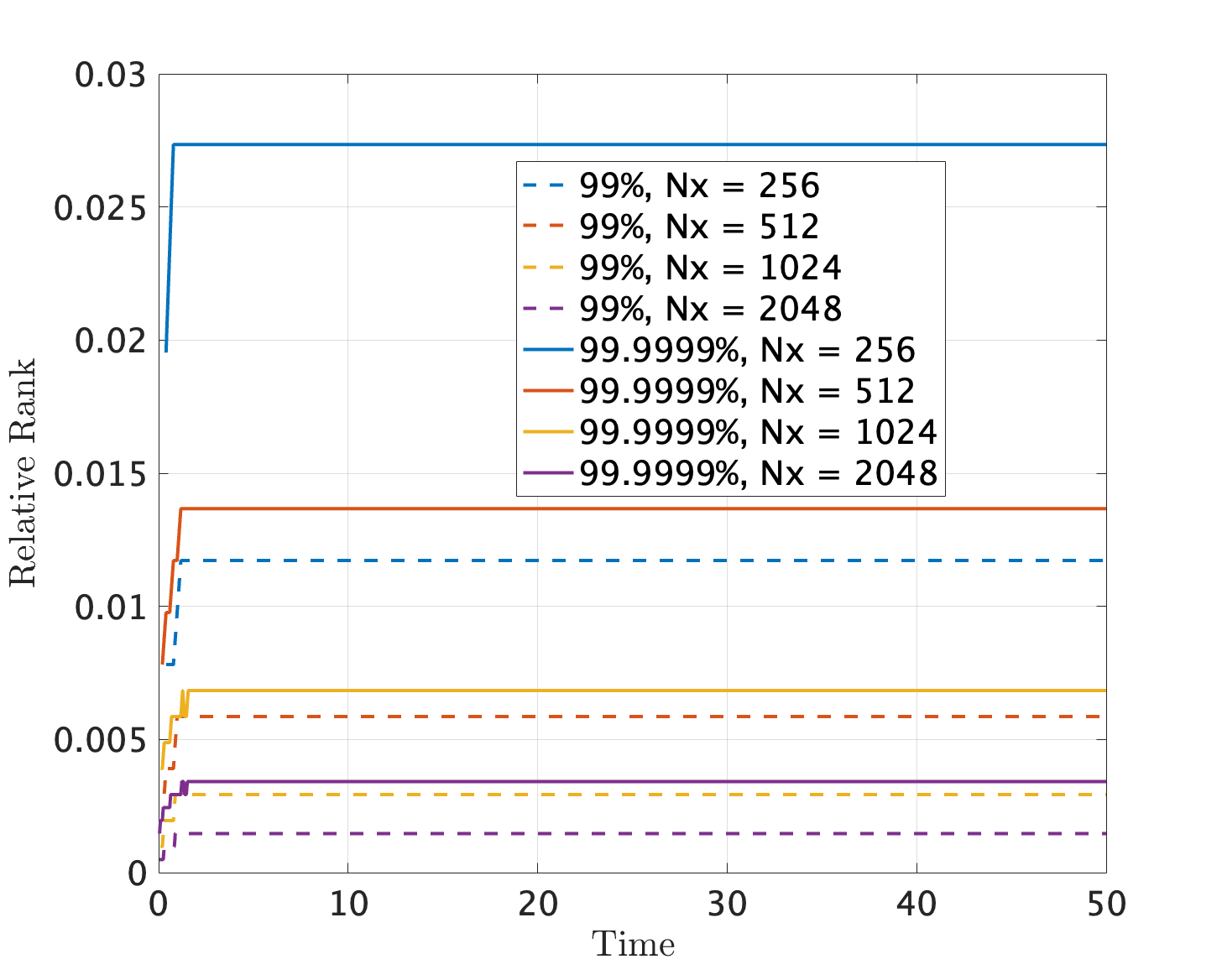}
        \caption{$H = 8$.}
        \label{fig:fullrank_H8_relaSVDrank_2048}
    \end{subfigure}  
    
   \caption{Time evolution of normalized ranks. This figure shows the time evolution of the SVD rank normalized by the mesh resolution, $N_x$. Subfigures (a)–(d) correspond to $H = 0.1$, $0.5$, $1$, and $8$, respectively. The normalized SVD rank is computed for energy thresholds of $99\%$ and $99.9999\%$ of the total energy in the full-rank solution. As \(N_x\) increases, the normalized rank required to capture the solution decreases, indicating improved compressibility at higher resolution. Additionally, the normalized rank itself decreases significantly as \(H\) increases, further supporting the emergence of low-rank structure in the moderate-to-high \(H\) regime.}
    \label{fig:fullrank_H_relaSVDrank}
\end{figure}

\section{A novel sampling-based adaptive rank method with operator splitting for the Wigner-Poisson System}\label{sec:Numerical methods}


{Motivated by the low-rank behavior observed in the previous section, we propose a sampling-based adaptive rank solver based on operator splitting for the Wigner–Poisson system. When applied to a model with low-rank structure, such a solver can effectively save computational cost. The method builds upon an improved full-rank solver, described in detail in Section~\ref{subsec:SSmethod}, where we enhance the Strang splitting framework \cite{arnold1996operator} with structure-preserving modifications to the Fourier update step. This structure-preserving full-rank solver was used to generate the numerical solutions presented in the previous section.}
In particular, we apply a Strang splitting strategy to decouple advection and nonlocal interaction terms, for which a conservative SL WENO scheme is applied for solving the advection subproblem (Section~\ref{subsec:SL}) and a structure-preserving Fourier-based approach is applied to handle the nonlocal term (Section~\ref{subsec:Fourier}). In Section~\ref{subsec:adaptive_rank_method}, we incorporate a sampling-based adaptive rank strategy, based on the adaptive cross approximation (ACA), to efficiently compress the numerical solution at each time step. The resulting scheme achieves a complexity of $\mathcal{O}(Nr^2+r^3)$ with $N$ being the mesh resolution per dimension and $r$ being the rank of the solution. To the best of our knowledge, this is the first adaptive rank approach for solving the Wigner-Poisson under the operator splitting framework.

\subsection{A full rank solver based on second order Strang splitting}\label{subsec:SSmethod}

The base method is the second-order Strang splitting solver first proposed in \cite{arnold1996operator}, with a modification of the scheme to achieve structure-preserving property in the Fourier update step (see Section \ref{subsec:Fourier}). Furthermore, the full-rank solver presented here lays the foundation for designing the adaptive-rank solver in Section \ref{subsec:adaptive_rank_method}. Let $n$ be an integer indicating the time step level and $\Delta t$ be a full time step, we describe the full rank solver using a semi-discrete solution $f(x,v,n\Delta t)$. 
Starting from $f(x,v,n\Delta t)$, the method first takes a half step in time, $\frac{\Delta t}{2}$, from $n \Delta t$ to $\left(n + \frac{1}{2}\right) \Delta t$, of advection,
\begin{align}\label{eq:split_adv1}
 \frac{\partial f}{\partial t} + v \frac{\partial f}{\partial x} &= 0,
\end{align}
with the resulting intermediate solution $f^{*,n+1/2}$. Then, it is followed by computing the density $\int dv\, f^{*,n+1/2}$ and solving for the potential,
\begin{align}\label{eq:split_poisson}
    -\frac{\partial^2 \Phi }{\partial x^2} = \int dv\, f^{*,n+1/2} - 1.
\end{align}
Under the frozen field approximation, meaning fields are held fixed, the method takes a full $\Delta t$ time  step 
of the Wigner operator with the initial condition $f^{*,n+1/2}$,
\begin{align}\label{eq:split_dbintegral}
    \frac{\partial f}{\partial t} &= - \frac{i }{2 \pi H^2} \int \int dv' dx' \exp\left( i \frac{v' - v}{H} x' \right) \left[ \Phi\left(x + \frac{x'}{2}\right) - \Phi\left(x - \frac{x'}{2}\right)\right] f(x,v', t),
\end{align}
with the result denoted by $f^{*,n+1}$. Finally, the method advances $f^{*,n+1}$ by a half time step $\Delta t/2$ to solve the advection subproblem~\eqref{eq:split_adv1}, and obtains the numerical solution $f^{n+1}$. Since we consider periodic problems, this approach can be easily extended to higher-order time splitting in higher dimensions \cite{gucclu2014arbitrarily}. In subsequent subsections, we describe the SL WENO scheme for advection and the Fourier approach for handling the Wigner potential operator.

In the following sections, unless otherwise specified, we consider a uniform grid on a periodic spatial domain $x \in [a,b]$, where the grid points are defined as $x_i = a + (i-1)\Delta x$, with $\Delta x = (b-a)/N_x$. The grid is ordered as
\begin{equation}\label{eqn:spatial_dis}
    a = x_1 < x_2 < \cdots < x_{N_x} = b - \Delta x.
\end{equation}
{Without loss of generality, we set $x \in [a,b] = [0, L_x]$ from now on.}
Similarly, we define a uniform grid on a symmetric velocity domain $v \in [-L_v, L_v]$, where the grid points are given by $v_i = -L_v + (i-1)\Delta v$, with $\Delta v = 2L_v / (N_v-1)$. The velocity grid satisfies
\begin{equation}\label{eqn:velocity_dis}
-L_v = v_1 < v_2 < \cdots < v_{N_v} = L_v.
\end{equation}
Let $f^{l}_{i,j}$ denotes the numerical solution at $(x_i,v_j)$ represented in a matrix form, where $l \in \{(*,n+1/2), (*,n+1),n,n+1\}$ represent different time level and {$\Phi^{n+1/2}_i$ represents the potential approximation at $(x_i,(n+1/2)\Delta t)$.} 

We note that the Poisson equation is solved using a standard fourth-order finite difference scheme~\cite{ames2014numerical}. For the right-hand side of the Poisson equation, the integration over velocity is computed using the trapezoidal rule on a symmetric, finite velocity domain. Since \(f(x,v)\) decays exponentially as \(|v| \to \infty\), the quadrature error is primarily determined by the domain truncation~\cite[Chap.~17]{boyd2001chebyshev}. For conservation of mass, we introduce an additional spatial integral in   \Cref{sec:full_rank_solver_summary}.
This spatial integration again uses a trapezoidal rule, however, as the integration is performed over a periodic domain, the method is spectrally accurate ~\cite[Page 458]{boyd2001chebyshev}. In implementation, both integrals reduce to simple discrete summations multiplied by \(\Delta v\) and \(\Delta x\), respectively.

\subsubsection{Conservative Semi-Lagrangian WENO Scheme}\label{subsec:SL}
We employ the conservative semi-Lagrangian WENO scheme introduced in~\cite{qiu2010conservative} to advance the advection steps in \eqref{eq:split_adv1}. 
To be self-contained, we briefly review the core idea of the method. {We use the update from $f^{n}$ to $f^{*,n+1/2}$ as an illustrative example; the update from $f^{*,n+1}$ to $f^{n+1}$ is handled similarly. 
{For clarity, we only present the third-order formulation for the case of a positive wave velocity and the negative velocity case can be derived by symmetry. 
In practice, we use the fifth-order scheme in our solver.}

The SL framework is based on the method of characteristics, where the advection equation is solved by following the characteristics backward in time given by characteristic equation 
$
\frac{dx}{dt} = v. 
$
{Let $f^n_j(x)$ denote the semi-discrete solution at time $t^n = n\Delta t, \, v = v_j$ and $f^{n}_{i,j} = f^n_j(x_i)$.}
For any fixed $v_{j}$, to compute $f^{*,n+1/2}_{i,j}$, one traces back along the characteristic to the departure point $x_d := x_i - v_j (\frac{\Delta t}{2})$ at time $t^{n+1/2} = (n+1/2)\Delta t$ with 
$f^{*,n+1/2}_{i,j} = f^n_j(x_d)$. 
{In particular, in the setting of a periodic domain, let $x_d = \mod\left(x_i - v_j (\frac{\Delta t}{2}),\, L_x\right)$. We start by building the conservative semi-Lagrangian WENO scheme for $v_j > 0$ and the case for $v_j < 0$ will be similar. For $v_j > 0$, if the shift distance $|v_j (\frac{\Delta t}{2})|$ is small, i.e., $|v_j (\frac{\Delta t}{2})| \le \frac{1}{2}\Delta x$, then $x_d \in [x_{i}-\frac{\Delta x}{2},x_i]$, and a left-biased (i.e., upwind for $v_j > 0$) piecewise cubic interpolation yields:
\begin{align}
f^{*,n+1/2}_{i,j} = f^n_j(x_d) &\approx f^n_{i,j} + \left(-\frac{1}{6} f^n_{i-2,j} + f^n_{i-1,j} - \frac{1}{2} f^n_{i,j} - \frac{1}{3} f^n_{i+1,j} \right) \xi + \left( \frac{1}{2} f^n_{i-1,j} - f^n_{i,j} + \frac{1}{2} f^n_{i+1,j} \right) \xi^2 \nonumber \\
&\quad + \left( \frac{1}{6} f^n_{i-2,j} - \frac{1}{2} f^n_{i-1,j} + \frac{1}{2} f^n_{i,j} - \frac{1}{6} f^n_{i+1,j} \right) \xi^3,
\end{align}
where $\xi = (x_d - x_{i}) / \Delta x \in [-\frac{1}{2},0]$. For larger shift distances, i.e., $|v_j (\Delta t / 2)| > \frac{1}{2}\Delta x$ and $v_j > 0$, define the index
\[
k = \mod\left(i - \left\lfloor \frac{v_j}{\Delta x} \left(\frac{\Delta t}{2}\right) + 0.5 \right\rfloor, N_x\right),
\]
so that $x_d \in [x_k - \frac{1}{2} \Delta x,\, x_k + \frac{1}{2} \Delta x)$ and $\xi = (x_d - x_k)/\Delta x \in [-\frac{1}{2},\frac{1}{2})$. Then, for $\xi \in [-\frac{1}{2},0]$, the interpolation becomes:
\begin{align}\label{eq:SL_k_intepolation}
f^{*,n+1/2}_{i,j} = f^n_j(x_d) &\approx f^n_{k,j} + \left(-\frac{1}{6} f^n_{k-2,j} + f^n_{k-1,j} - \frac{1}{2} f^n_{k,j} - \frac{1}{3} f^n_{k+1,j} \right) \xi \nonumber \\
&\quad + \left( \frac{1}{2} f^n_{k-1,j} - f^n_{k,j} + \frac{1}{2} f^n_{k+1,j} \right) \xi^2 \nonumber \\
&\quad + \left( \frac{1}{6} f^n_{k-2,j} - \frac{1}{2} f^n_{k-1,j} + \frac{1}{2} f^n_{k,j} - \frac{1}{6} f^n_{k+1,j} \right) \xi^3,
\end{align}}

\noindent for $\xi \in (0,\frac{1}{2})$, the stincel shifts be one to the right.  
To build an upwind WENO reconstruction, we rewrite the upwind polynomial as
\begin{equation}\label{eq:SL_update}
  f^{*,n+1/2}_{i,j} = f^n_{k,j} - \xi \left( \hat{f}^n_{k+1/2,j}(\xi) - \hat{f}^n_{k-1/2,j}(\xi) \right), 
\end{equation}
where $\hat{f}^n_{k+1/2,j}(\xi)$ and $\hat{f}^n_{k-1/2,j}(\xi)$ are in the form of fluxes. 
The fluxes $\hat{f}^n$ are determined by $f^n$ and depend on order of the scheme and the nonlinear WENO weights. 
{For the linear weights and $\xi \in [-\frac{1}{2},0]$, \eqref{eq:SL_update} must exactly match \eqref{eq:SL_k_intepolation}. For the remaining case $\xi \in (0,\frac{1}{2})$, corresponding to $x_d \in (x_k, x_k + \frac{1}{2}\Delta x)$, the upwind stencils are shifted one index to the right; see Figure \ref{fig:SLdespcription} for a visualization of stencil choices.}
In particular, the flux function $\hat{f}^n_{k-1/2,j}(\xi)$ will be expressed as:
\begin{equation*}
    \hat{f}^n_{k-1/2,j}(\xi) = 
    \begin{cases}
        (f^n_{k-2,j},f^n_{k-1,j},f^n_{k,j}) \, C^L_3 \, (1, |\xi|, \xi^2)^{\mathrm{T}}, \quad \xi \in [-\frac{1}{2},0]\\
        (f^n_{k-1,j},f^n_{k,j},f^n_{k+1,j}) \, C^L_3 \,(1, |\xi|, \xi^2)^{\mathrm{T}}, \quad \xi \in (0,\frac{1}{2}).
    \end{cases}
\end{equation*}
Here, the linear coefficient matrix $C^L_3$ is given by
\begin{equation}\label{eq:SL_linearmatrix}
    C^L_3 = \begin{bmatrix} -\frac{1}{6} & 0 & \frac{1}{6} \\ \frac{5}{6} & \frac{1}{2} & -\frac{1}{3} \\ \frac{1}{3} & -\frac{1}{2} & \frac{1}{6} \end{bmatrix} = \begin{bmatrix} -\frac{1}{2}\omega_1 & 0 & \frac{1}{6} \\ \frac{3}{2}\omega_1 + \frac{1}{2}\omega_2 & \frac{1}{2} & -\frac{1}{3} \\ \frac{1}{2}\omega_2 & -\frac{1}{2} & \frac{1}{6} \end{bmatrix}, 
\end{equation}
with linear weights $\omega_1 = \frac{1}{3}, \, \omega_2 = 1-\omega_1$. 
We now replace $C^L_3$ with a WENO-type nonlinear weights matrix $\tilde{C}^L_3$.  In the SL WENO context, it is important to note that it is sufficient to only apply WENO to the leading order term in the polynomial reconstruction.
Similar to the standard WENO scheme, we compute the smoothness indicators based on the chosen stencils:
 \[
\beta_1 = (f^n_{k-1,j}-f^n_{k-2,j})^2, \quad \beta_2 = (f^n_{k-1,j}-f^n_{k,j})^2.
\]
We define the nonlinear weights as ${\omega}'_1 = \omega_1\frac{1}{(\varepsilon+\beta_1)^2}$, ${\omega}'_2 = \omega_2 \frac{1}{(\varepsilon+\beta_2)^2}$. Then, we normalize them
\[
\tilde{\omega}_1 = \frac{{\omega}'_1}{{\omega}'_1+{\omega}'_2}, \quad \tilde{\omega}_2 = \frac{{\omega}'_2}{{\omega}'_1+{\omega}'_2},
\]
and substitute into \eqref{eq:SL_linearmatrix} to construct the WENO-type weights matrix $\tilde{C}^L_3$.
{We have now constructed the third-order SL WENO scheme for positive velocity $v_j > 0$. Similarly, for negative velocity $v_j < 0$, the third-order SL WENO scheme is constructed by a symmetric flip of the points used in the approximation for $v_j > 0$.}

In this paper, we use the fifth-order conservative SL WENO scheme to solve equations \eqref{eq:split_adv1}; the fifth order spatial accuracy significantly mitigates numerical diffusion with relatively coarse computational meshes \cite{gucclu2014arbitrarily}. 
\begin{figure}[htbp]
  \centering
  \includegraphics[width=0.6\textwidth]{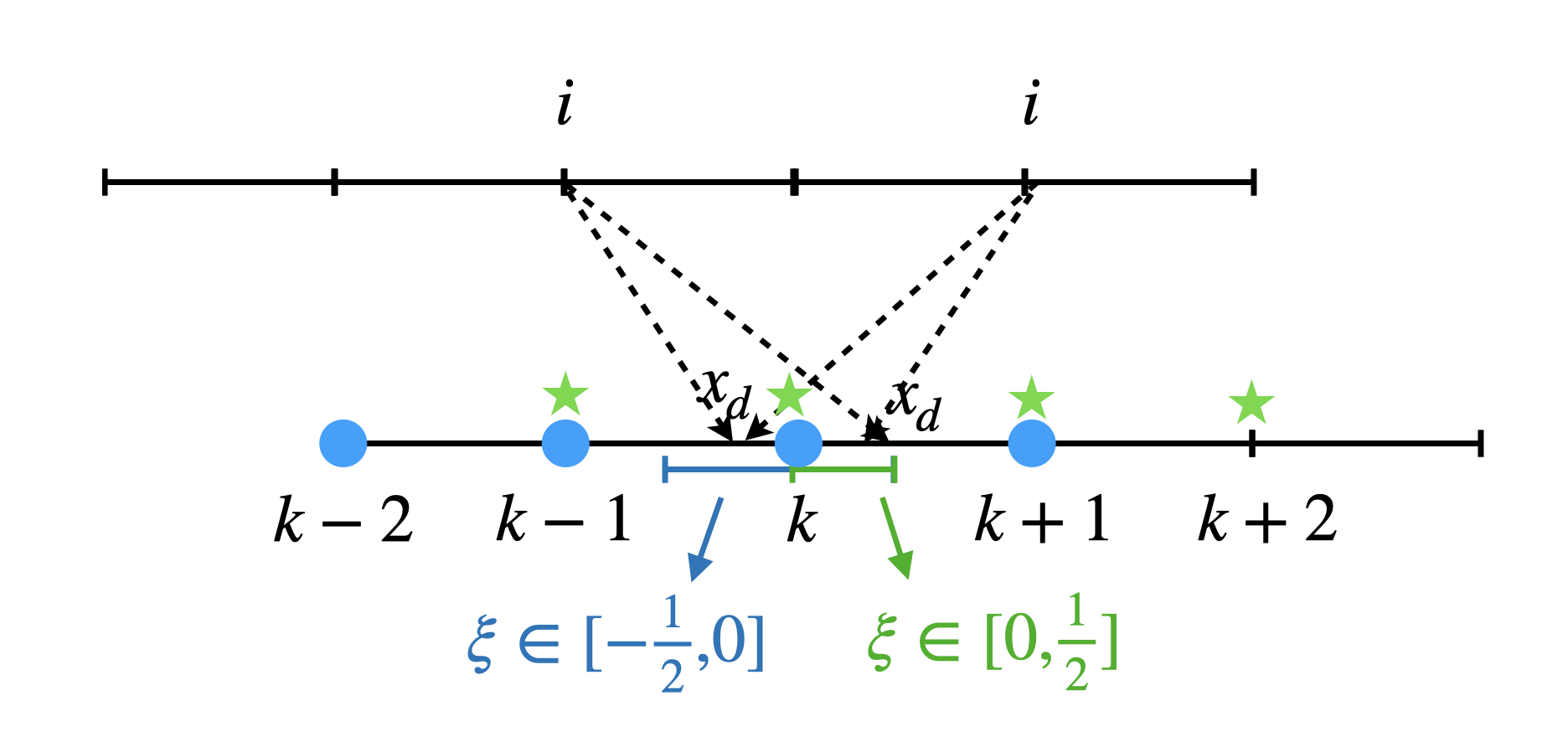}
  \caption{Description of the stencil choices for the Semi-Lagrangian WENO scheme: the first line represents the mesh at the current time level, where the solution is sought at the $i$-th grid point $x_i$; the second line corresponds to the mesh at the previous time level. For positive (negative) velocity $v_j$, the departure point $x_d$ is obtained by tracing backward along the characteristic line (dashed line) from $x_i$ to the left (right). Two different upwind stencil choices for the interpolation for different $\xi$ are indicated using blue circles and green stars, respectively.}
  \label{fig:SLdespcription}
\end{figure}

\subsubsection{Structure-preserving Fourier update}\label{subsec:Fourier}

In this subsection, we discuss our strategy for solving the subproblem~\eqref{eq:split_dbintegral}, with the goal of deriving an analytic update formula, 
upon which a discretization will be applied to obtain the 
numerical scheme. 
The challenge associated with the subproblem ~\eqref{eq:split_dbintegral} is the right hand side nonlocal term, which involves a double integral. To handle this effectively, we employ a Fourier transform in velocity space and derive a set of decoupled ODEs, that are linearized from fixing the field from the previous step of the Strang splitting. In the following, we first derive a closed-form expression in the continuous setting and discuss its inherent symmetry. We then explain how to preserve this structure when applying its discrete counterpart via the fast Fourier transform (FFT). 


Recall that the Fourier transform and its inverse are defined as
\[
{\tilde{f}(\omega):=}
\mathcal{F}\{f(t)\} = \int_{-\infty}^\infty f(t)\, e^{-i \omega t} \, dt, \quad 
\mathcal{F}^{-1}\{\tilde{f}(\omega)\} = \frac{1}{2\pi} \int_{-\infty}^\infty \tilde{f}(\omega)\, e^{i \omega t} \, d\omega.
\]
Applying the Fourier transform in \(v\) to~\eqref{eq:split_dbintegral}, let $k_v$ be the Fourier co we focus on the key term
\[
\mathcal{F}_v\left\{ \int \exp\left(i \frac{v' - v}{H} x'\right) f(x,v',t)\, dv' \right\}
= \mathcal{F}_v\left\{ e^{-i \frac{v}{H} x'} \right\} \cdot \mathcal{F}_v\{f(x,v,t)\} 
= 2\pi\, \delta\left(k_v + \frac{x'}{H}\right) \tilde{f}(x,k_v,t),
\]
{where $\tilde{f}$ is the Fourier transform of the solution $f(x,v,t)$ in the velocity space.}
Substituting this into the equation gives
\[
\frac{\partial \tilde{f}}{\partial t} 
= -\frac{i}{H^2} \int dx'\, \delta\left(k_v + \frac{x'}{H}\right) 
\left[ \Phi\left(x + \frac{x'}{2}\right) - \Phi\left(x - \frac{x'}{2}\right) \right] \tilde{f}(x,k_v,t).
\]
We now introduce the change of variable \(\tilde{x} = k_v + \frac{x'}{H}\), so that \(dx' = H d\tilde{x}\). This yields
\begin{equation}\label{eq:ODE_kv}
\begin{split}
    \frac{\partial \tilde{f}}{\partial t} 
    &= -\frac{i}{H} \int d\tilde{x}\, \delta(\tilde{x}) 
    \left[ \Phi\left(x + \frac{H(\tilde{x} - k_v)}{2} \right) 
         - \Phi\left(x - \frac{H(\tilde{x} - k_v)}{2} \right) \right] \tilde{f}(x,k_v,t) \\
    &= \frac{i}{H} \left[ \Phi\left(x + \frac{H k_v}{2} \right) - \Phi\left(x - \frac{H k_v}{2} \right) \right] \tilde{f}(x,k_v,t),
\end{split}
\end{equation}
where we have used the identity \(\int \delta(x) f(x)\, dx = f(0)\).

Equation~\eqref{eq:ODE_kv} is a linear ODE in time and admits the exact solution over a full time step \(\Delta t\) as
\begin{equation}\label{eq:FFT_updateform}
\tilde{f}(x,k_v,t_0+\Delta t) = \tilde{f}(x,k_v,t_0) 
\exp \left( \frac{i}{H} \left[ \Phi\left(x + \frac{H k_v}{2} \right) - \Phi\left(x - \frac{H k_v}{2} \right) \right] \Delta t \right), 
\end{equation}
where \(\tilde{f}(x,k_v,t_0)\) is taken from the Fourier transform of $f^{*,n+1/2}$ with respect to velocity space, denoted by $\tilde{f}^{*,n+1/2}$.
Equation~\eqref{eq:FFT_updateform} reveals an important structural property: the update involves multiplying the Fourier transform of a real-valued function by a conjugate symmetric complex function,
\begin{equation}\label{eqn:expo_g}
    g^x(k_v) := \exp \left( \frac{i}{H} \left[\Phi\left(x + \frac{H k_v}{2}\right) - \Phi\left(x - \frac{H k_v}{2}\right)\right] \Delta t \right),
\end{equation}
which satisfies
\[
g^x(-k_v) = \overline{g^x(k_v)}, \quad \forall k_v \in \mathbb{R} \text{ or } \mathbb{Z}.
\]
By standard Fourier theory, the transform of any real-valued function is conjugate symmetric, and conversely, the inverse transform of a conjugate symmetric function is guaranteed to be real. As a result, the updated solution should remain real-valued in exact arithmetic. However, preserving this structure in numerical practice requires careful handling. For completeness, we present a definition of the structure-preserving property in the context of this paper. 
\begin{definition}[Structure-Preserving Scheme]\label{def:structure-preserving}
    A numerical scheme is said to be structure-preserving if it retains key analytical or physical properties of the continuous system at the discrete level. 
\end{definition}
In the context of our method, the scheme is structure-preserving in the sense that with input real-valued function ${f}(x,v,t_0) $, the inverse Fourier transform of $\tilde{f}(x,k_v,t_0 + \Delta t)$ in \eqref{eq:FFT_updateform} remains real-valued at each time step. 
Furthermore, in the full-rank scheme, with real-valued phase function $f^{*,n+1/2}$, the updated solution of the full-rank solver $f^{*,n+1}$ from \eqref{eq:split_dbintegral} remains real-valued. We will revisit this structure preserving property when we propose a modified sampling-based adaptive rank algorithm, which will enforce this property. 

Next, we describe our proposed discretization and discuss its structure-preserving property as introduced above. 
    We discretize the velocity space $[-L_v,L_v]$ with $N_v$ points for the discrete Fourier transform (DFT) \cite{boyd2001chebyshev}. 
    Here, we choose \(N_v\) to be even when designing the scheme, since the FFT algorithm is more efficient for vectors whose length is a power of two \cite{knuth2014art}.
    For an $N_v$-point real sequence $f[n]$, the standard formula of DFT is given by:
\begin{equation}
\tilde{f}[k_v] = \sum_{n=0}^{N_v-1} f[n] e^{-i\frac{2\pi}{N_v}k_v n}, \quad 
f[n] = \frac{1}{N_v} \sum_{k_v=0}^{N_v-1} \tilde{f}[k_v] e^{i\frac{2\pi}{N_v}k_v n},
\end{equation}
with discrete frequency indices $k_v\in\{0,...,N_v-1\}$. To preserve the structure-preserving property, the frequency domain of the DFT is shifted to be zero-centered and normalized by \(\frac{2\pi}{2L_v}\),
\begin{equation}\label{eq:frequency space}
K_v = \frac{2\pi}{2L_v} \left\{-\frac{N_v}{2}, -\frac{N_v}{2}+1, \dots, 0, \dots, \frac{N_v}{2}-1 \right\}.
\end{equation}
Let $k_{Nyq} = \frac{2\pi}{2L_v} (-\frac{N_v}{2})$, which is often called the Nyquist frequency.   
For \( f \in \mathbb{R}^{N_v} \) with even \(N_v\), the corresponding DFT coefficients \(\tilde{f}[k_v]\) satisfy the conjugate symmetry property \(\tilde{f}[-k_v] = \overline{\tilde{f}[k_v]}\) for \(k_v \in K_v \setminus \{0, k_{\text{Nyq}}\}\), with both \(\tilde{f}[0]\) and \(\tilde{f}[k_{\text{Nyq}}]\) being real.
The update operator \(g(k_v)\) acts on each frequency mode and preserves conjugate symmetry except at the Nyquist frequency. At \(k_v = k_{\text{Nyq}}\), the phase term 
\begin{equation}
\Phi\left(x + \frac{-H N_v \pi}{4L_v}\right) - \Phi\left(x - \frac{-H N_v \pi}{4L_v}\right)
\label{eq: phi}    
\end{equation}
generally makes \(g(k_v)\) complex, thereby breaking the condition \(\tilde{f}[k_{\text{Nyq}}] \in \mathbb{R}\), and introducing small imaginary components in the inverse DFT. 
To recover a real solution, we set \(\tilde{f}(x, k_v, t_0 + \Delta t) = 0\) at \(k_v = k_{\text{Nyq}}\) before performing the inverse transform. This truncation does not degrade overall accuracy, since the Nyquist mode typically carries negligible energy for smooth solutions. A schematic illustration of this correction is shown in Figure~\ref{fig:Fourierdespcription}.
Note that the potential in \eqref{eq: phi} needs to be evaluated at locations \(x \pm \frac{H k_v}{2}\), which in general do not coincide with the spatial grid. To address this, we apply a WENO reconstruction to interpolate \(\Phi\) at these off-grid locations. 
For a fixed spatial grid point $x_i$, this involves a local reconstruction of \(\Phi\) for each $k_v$, where the interpolation points depend on the displacement \(\frac{H k_v}{2}\).
\begin{figure}[htbp]
  \centering
  \includegraphics[width=0.6\textwidth]{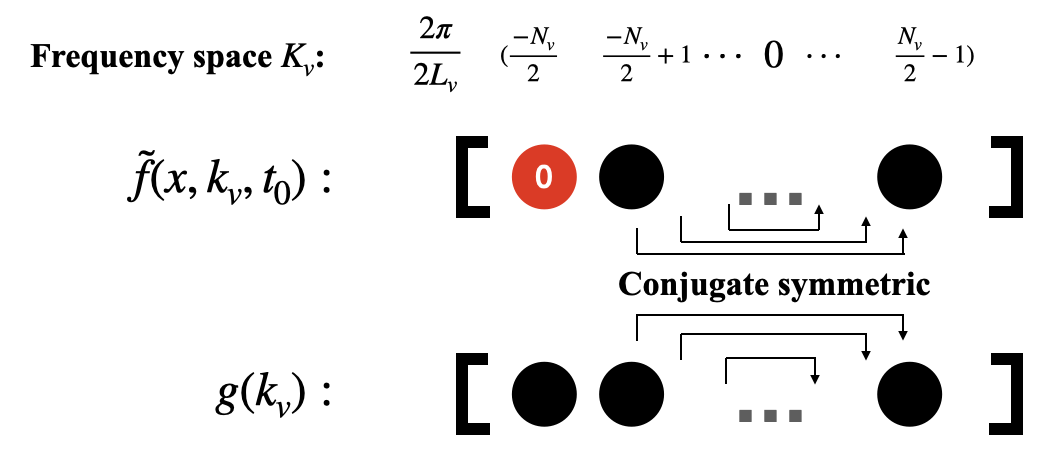}
  \caption{Description of the structure-preserving Fourier update: The first line shows the frequency order (without normalized). The second line shows the FFT structure of a real-valued vector with truncation of the Nyquist frequency, preserving conjugate symmetry for the other parts. The third line shows the structure of the exponential vector.}
  \label{fig:Fourierdespcription}
\end{figure}

\subsubsection{Summary of the Algorithm}\label{sec:full_rank_solver_summary}
We summarize the improved full-rank solver in Algorithm~\ref{alg:fullrank-WP solver}, which incorporates all the components described above except for the final correction step. Although a conservative semi-Lagrangian method is employed for the advection subproblem, the Fourier update may still lead to a loss of mass conservation. To restore mass conservation, we introduce a Step 5, where a Lagrange multiplier is applied to enforce mass preservation. We remark that Arnold and Ringhofer conducted a stability analysis for the original full-rank solver in \cite{arnold1995operator, arnold1996operator}; in particular, their 1996 paper shows that the proposed splitting method—also employed here—is stable and achieves second-order accuracy in time. Our improved full-rank solver retains this stability property, as the modifications introduced in the Fourier update step preserve the overall accuracy and do not compromise the stability of the scheme. 
\begin{algorithm}
\caption{Full rank Wigner-Poisson Solver}
\label{alg:fullrank-WP solver}
\begin{algorithmic}[1]
\Require Perform the spatial discretization \eqref{eqn:spatial_dis} and velocity discretization \eqref{eqn:velocity_dis}. And recall the notations: $f^{l}_{i,j}$ denotes the numerical solution at $(x_i,v_j)$ represented in a matrix form, where $l \in \{(*,n+1/2), (*,n+1),n,n+1\}$ represents different time level and {$\Phi^{n+1/2}_i$ represents the potential approximation at $(x_i,(n+1/2)\Delta t)$.} 
\While{$t < T$} 
    \State \textbf{Step 1: Semi-Lagrangian Method}
    \State Take WENO approximation of half-step update (Section \ref{subsec:SL}) and in particular, for any fixed velocity grid point $v_j$, use the formula \eqref{eq:SL_update} to get $f^{*,n+1/2}_{i,j}$: 
    \begin{equation}\label{eqn:SL_matrix}
        f^{*,n+1/2}_{i,j} = f^n_{i,j} - \xi \left( \hat{f}^n_{i+1/2,j}(\xi) - \hat{f}^n_{i-1/2,j}(\xi) \right).
    \end{equation}
    \State \textbf{Step 2: Poisson Equation Step}
    \State Use trapezoid rule for $f^{*,n+1/2}$ along velocity space to compute the source term. 
    \State Solve the Poisson equation \eqref{eq:poisson} with a fourth-order finite difference method, periodic boundary conditions, and zero-mean density. This leads to $\Phi^{n+1/2}$. 
    \State \textbf{Step 3: Fourier-ODE Step}
    \State Apply FFT to $f^{*,n+1/2}$ along velocity space to get $\tilde{f}^{*,n+1/2}$;
    \State Solve ODE in Fourier space with \eqref{eq:FFT_updateform} and \eqref{eqn:expo_g} and specifically,  
    \begin{equation}\label{eqn:FFT_matrix}
        \tilde{f}^{*,n+1}_{i,j} = \tilde{f}^{*,n+1/2}_{i,j}g^{x_i}(k_v),
    \end{equation}
    where $k_v = \frac{2\pi}{2L_v} (-\frac{N_v}{2}+j+1)$ and the reconstruction of {$\Phi^{n+1/2}$} is used in \eqref{eqn:expo_g}. 
    \State \textbf{Step 4: Second Semi-Lagrangian Step}
    \State Same as Step 1 but we will use $f^{*,n+1}$ instead of $f^n$. This step leads to $f^{n+1}$ without mass-conservation. 
    \State The mass is computed by trapezoid rule along velocity space and simple sums along spatial space: 
    \[
L^{n+1,*} = \sum_{i} \sum_{j}f^{n+1}(i,j) \, \Delta x \, \Delta v.
\]
    \State \textbf{Step 5: Density Correction Step}
    \State Use a Lagrange multiplier: 
    \[
f^{n+1} = \frac{L_x}{L^{n+1,*}}f^{n+1}.
\]
    \State $t \gets t + \Delta t$
\EndWhile
\end{algorithmic}
\end{algorithm}

\subsection{A sampling-based adaptive rank method}\label{subsec:adaptive_rank_method}



To reduce the computational cost, we seek to represent the solution snapshots in a low-rank form at each time step. In the adaptive rank framework, both the SL WENO advection step and the Fourier update step are reformulated as matrix approximation problems, where the solution matrix is compressed using the ACA-SVD strategy introduced in Section~\ref{sec:ACA-SVD}. The ACA stage adaptively samples informative rows and columns, while the subsequent SVD stage further stabilizes the approximation and removes redundant modes. This sampling-based compression procedure serves as the core building block for the adaptive rank Wigner–Poisson solver described in Section~\ref{subsec:NewLR_WP}. In additional, during the Fourier update stage, we incorporate a modified structure-preserving ACA sampling strategy to maintain the real-valued structure of the solution, as discussed in Section~\ref{subsec:structure-preserving}.

\subsubsection{ACA-SVD Compression Strategy}\label{sec:ACA-SVD}

In our Wigner–Poisson solver, each solution snapshot can be naturally viewed as a matrix of values defined on a Cartesian grid. We compress these snapshots using a  Adaptive Cross Approximation (ACA), followed by Singular Value Decomposition (SVD) type truncation. The key motivation for employing ACA is its ability to construct a low-rank approximation by adaptively sampling only a small subset of the matrix entries. This feature is particularly valuable in our setting, as it allows the algorithm to invoke the numerical solver selectively at the required grid locations (rows and columns), without needing to assemble the full matrix. In contrast, a direct SVD would require access to all matrix entries, which is computationally impractical when the entire matrix is not explicitly constructed. After the ACA stage, a SVD type truncation is applied to further eliminate redundancy and enhance numerical stability, yielding the final low-rank representation. A more detailed discussion of such procedure can also be found in~\cite{zheng2025semi}.

The ACA procedure produces a CUR decomposition of the form
\begin{equation}\label{eq:CUR_decom}
A_k = A(:, \mathcal{J})\, A(\mathcal{I}, \mathcal{J})^{-1}\, A(\mathcal{I}, :) =: \mathcal{C}\mathcal{U}\mathcal{R},
\end{equation}
where \(A \in \mathbb{R}^{N_x \times N_v}\) is the matrix to be compressed; \(\mathcal{I} = \{i_1, i_2, \dots, i_k\}\) and \(\mathcal{J} = \{j_1, j_2, \dots, j_k\}\) denote the selected row and column index sets; and \(\mathcal{C}\), \(\mathcal{R}\), and \(\mathcal{U}\) are the column, row, and core matrices defining the approximation. The overall CUR decomposition structure is illustrated on the right side of Figure~\ref{fig:CURdespcription}. While finding the optimal index sets for CUR is NP-hard~\cite{civril2007finding}, ACA empirically constructs \(\mathcal{I}\) and \(\mathcal{J}\) adaptively via a greedy pivoting process driven by residual evaluation.

At each iteration \(k\), starting from an initial approximation \(A_0 = 0\), the ACA algorithm recursively constructs a rank-\(k\) approximation \(A_k\) via a rank-one update; which consists of two stages. Please see Figure~\ref{fig:CURdespcription} for a schematic illustration of the procedure. 
\paragraph{Phase I: Pivot selection.} 
A candidate set of index pairs \(\mathcal{P} = \{(i^*_\ell, j^*_\ell)\}_{\ell=1}^{p}\) (blue dots in Figure~\ref{fig:CURdespcription}) is randomly sampled from unselected rows and columns of the residual matrix from the previous iteration $R_{k-1}=A - A_{k-1}$. The initial pivot is selected as
\[
(i^*, j^*) = \arg\max_{(i,j) \in \mathcal{P}} |R_{k-1}(i,j)|.
\]
This is followed by greedy refinement (dashed green column and solid yellow row in Figure~\ref{fig:CURdespcription}):
\[
i_k = \arg\max_i |R_{k-1}(i,j^*)|, \qquad j_k = \arg\max_j |R_{k-1}(i_k,j)|.
\]
The selected \((i_k, j_k)\) pair is then added to \(\mathcal{I}\) and \(\mathcal{J}\).
\paragraph{Phase II: Rank-one update.}
The approximation is updated via
\begin{equation}\label{eq:recursive_update}
A_k = A_{k-1} + \frac{1}{R_{k-1}(i_k, j_k)}\, R_{k-1}(:, j_k)\, R_{k-1}(i_k, :),
\end{equation}
which is mathematically equivalent to the standard CUR decomposition \eqref{eq:CUR_decom}~\cite{shi2024distributed}. Such construction interpolates the matrix \(A\) at all previously selected rows and columns:
\[
A_k(i_\ell,:) = A(i_\ell,:), \qquad A_k(:,j_\ell) = A(:,j_\ell), \qquad \text{for } \ell = 1,\dots,k.
\]
The iteration terminates when either the norm of the most recent rank-one update falls below a prescribed tolerance \(\varepsilon_c\), or a maximum allowable rank is reached.

The resulting ACA approximation admits a structured factorization:
\[
A_k = \mathcal{E}_{\mathcal{J}} \, \mathcal{D} \, \mathcal{E}_{\mathcal{I}},
\]
where \(\mathcal{E}_{\mathcal{J}} = [\mathbf{c}_1, \dots, \mathbf{c}_k]\) with \(\mathbf{c}_\ell = R_{\ell-1}(:, j_\ell)\), \(\mathcal{E}_{\mathcal{I}} = [\mathbf{r}_1, \dots, \mathbf{r}_k]^{\mathrm{T}}\) with \(\mathbf{r}_\ell = R_{\ell-1}(i_\ell, :)^{\mathrm{T}}\), and \(\mathcal{D} = \mathrm{diag}(r_1^{-1}, \dots, r_k^{-1})\) with \(r_\ell = R_{\ell-1}(i_\ell, j_\ell)\).

\paragraph{SVD truncation.}
To eliminate numerical redundancy and enhance stability, the ACA output is further compressed via an SVD-type truncation step. We first compute QR decompositions:
\[
\mathcal{E}_{\mathcal{J}} = Q_1 R_1, \qquad \mathcal{E}_{\mathcal{I}}^{\mathrm{T}} = Q_2 R_2.
\]
We then compute the SVD of the reduced core:
\[
R_1 \mathcal{D} R_2^{\mathrm{T}} = \widetilde{U} \widetilde{\Sigma} \widetilde{V}^{\mathrm{T}}.
\]
The final approximation is truncated at rank \(r_s\), defined as the smallest integer such that \(\widetilde{\Sigma}(r_s+1,r_s+1) < \varepsilon_s\). The resulting factorization is
\[
A \approx \left( Q_1 \widetilde{U}(:,1\!:\!r_s) \right) \widetilde{\Sigma}(1\!:\!r_s,1\!:\!r_s) \left( Q_2 \widetilde{V}(:,1\!:\!r_s) \right)^{\mathrm{T}} =: U \Sigma V^{\mathrm{T}}.
\]

\begin{figure}[htbp]
  \centering
  \includegraphics[width=0.84\textwidth]{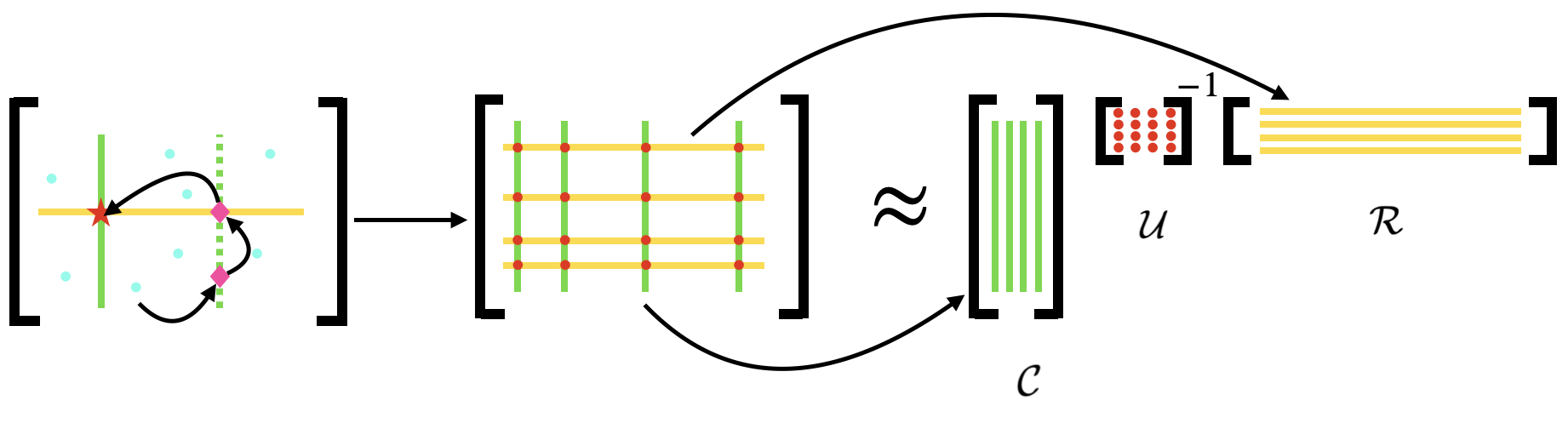}
  \caption{Schematic illustration of ACA construction and CUR decomposition. The left panel illustrates one ACA iteration step: a rank-one update is generated by finding pivot row and column indices based on a greedy procedure to search for maximum residual. The right panel shows the CUR factorization structure assembled by ACA after multiple iterations. Randomly sampled candidate entries are shown as blue dots. The selected columns and rows are indicated by green and yellow lines, respectively. The red dots mark the intersection points of the selected rows and columns, which form the submatrix $A(\mathcal{I},\mathcal{J})$ used to construct the core matrix $\mathcal{U}$ in the CUR decomposition.}
  \label{fig:CURdespcription}
\end{figure}

\begin{remark}
The computational cost of the above ACA-SVD procedure is mainly due to two procedures: (1) adaptive sampling of matrix entries, and (2) the SVD-based truncation. Assume that accessing each matrix entry requires \(\mathcal{O}(1)\),  the total cost of ACA sampling up to final rank \(r\) is \(\mathcal{O}(Nr)\),
where \(N = \max\{N_x, N_v\}\). 
The SVD based truncation incurs \(\mathcal{O}(Nr^2 + r^3)\) complexity, accounting for two QR factorizations of size \(N \times r\) and one SVD on the reduced \(r \times r\) core. These lead to a total computational complexity that scales linearly with \(N\), provided that the numerical rank \(r\) remains moderate and is independent of $N$.
\end{remark}

\begin{remark}
The initial sample size \(p\) used in {Phase I} of ACA is a user-defined parameter. Larger values of \(p\) typically improve the robustness and accuracy of the ACA procedure. In our implementation, we set \(p = 12\). The refinement strategy also differs from the classical ACA formulation in~\cite[Algorithm 2]{dolgov2020parallel}: we do not impose the rook pivoting condition, which would require selecting the global maximum within both the corresponding row and column. Such a condition is computationally prohibitive in our setting, as accessing individual matrix entries requires invoking local numerical PDE solvers. In practice, we find that searching a pivot along a column, followed by a row yields stable and accurate results.
\end{remark}

\subsubsection{New adaptive-rank Wigner–Poisson solver}
\label{subsec:NewLR_WP}

We now integrate all previously introduced components to construct the adaptive-rank Wigner–Poisson solver, which is presented in Algorithm~\ref{alg:LR-WP solver}. 
The method is built upon the full-rank solver presented in Section~\ref{subsec:SSmethod} (Algorithm~\ref{alg:fullrank-WP solver}), but reformulates each substep into an adaptive-rank framework by applying the ACA-SVD compression strategy introduced in Section~\ref{sec:ACA-SVD}. The numerical solution is represented in SVD form $\mathbf{F}^l = U^l \Sigma^l (V^l)^\mathrm{T}$ for $l \in \{(*,n+1/2), (*,n+1),n,n+1\}$, where $\Sigma^l$ is diagonal. 
In addition, a structure-preserving modification is incorporated into the ACA stage of the Fourier update to ensure real-valued solutions, \Cref{subsec:structure-preserving}. 

Below we highlight essential components of Algorithm~\ref{alg:LR-WP solver}:
\begin{enumerate}
    \item \textbf{Step 1} combines the full rank SL WENO scheme with the ACA-SVD compression. The function handle to access the matrix entry is defined by evaluating the WENO update formula~\eqref{eqn:SL_matrix} at the corresponding characteristic trace back point. The ACA-SVD algorithm described in Section~\ref{sec:ACA-SVD} is applied to obtain $\mathbf{F}^{*,n+1/2} = U^{*,n+1/2} \Sigma^{*,n+1/2} (V^{*,n+1/2})^\mathrm{T}$ in a low rank decomposition format. 

    \item \textbf{Step 2} solves the Poisson equation using the same method as in Algorithm~\ref{alg:fullrank-WP solver}, but with the source term computed with low rank complexity from the SVD representation $\mathbf{F}^{*,n+1/2}$.

    \item \textbf{Step 3} performs the Fourier update under adaptive rank framework. First, the FFT is applied column-wise to the right singular factor $V^{*,n+1/2}$, yielding $\tilde{\mathbf{F}}^{*,n+1/2} = U^{*,n+1/2} \Sigma^{*,n+1/2} \left( \mathcal{F}(V^{*,n+1/2}) \right)^{\mathrm{T}}$.
    The function handle to access the matrix is defined using  ~\eqref{eqn:FFT_matrix}.
    Then the ACA-SVD algorithm is applied to obtain the Fourier update $\tilde{\mathbf{F}}^{*,n+1} = U^{*,n+1} \Sigma^{*,n+1} \left(\tilde{V}^{*,n+1} \right)^{\mathrm{T}}$.
    To preserve conjugate symmetry, the procedure of the ACA column selection is modified: for each selected column index $j$, its symmetric counterpart
    \begin{equation}\label{eq:symmetric_column_pair}
    j_{\text{opp}} = \mod((N_v + 1 - j), N_v) + 1    
    \end{equation}
    is also included. This guarantees that the resulting CUR approximation respects conjugate symmetry, as proven in Theorem~\ref{thm:structure-preserving}. 
    After we construct $\tilde{\mathbf{F}}^{*,n+1}$, the inverse FFT is applied column-wise to the right singular factor $\tilde{V}^{*,n+1}$, giving the  
    real-valued representation $\mathbf{F}^{*,n+1}$ (up to machine precision):
    \[
    \mathbf{F}^{*,n+1} = U^{*,n+1} \Sigma^{*,n+1} \left( \mathcal{F}^{-1}(\tilde{V}^{*,n+1}) \right)^{\mathrm{T}}.
    \]
    \item \textbf{Step 4} applies another ACA-SVD semi-Lagrangian update, starting from $\mathbf{F}^{*,n+1}$ to compute $\mathbf{F}^{n+1} = U^{n+1} \Sigma^{n+1} (V^{n+1})^{\mathrm{T}}$. The total mass is directly computed from its SVD factors:
    \begin{equation}\label{eq:total_mass_low_rank}
    L^{n+1,*} = \sum_{k=1}^{r} \Sigma(k,k) \left( \sum_{i=1}^{N_x} U(i,k) \Delta x \right) \left( \sum_{j=1}^{N_v} V(j,k) \Delta v \right).
    \end{equation}
    
    \item \textbf{Step 5} applies a Lagrange multiplier correction to enforce mass conservation:
    \[
    \mathbf{F}^{n+1,C} = U^{n+1} \left( \Sigma^{n+1} \cdot \frac{L_x}{L^{n+1,*}} \right) (V^{n+1})^\mathrm{T}.
    \]
\end{enumerate}

\begin{algorithm}
\caption{Algorithm: Adaptive-Rank Wigner–Poisson Solver with ACA-SVD Compression}
\label{alg:LR-WP solver}
\begin{algorithmic}[1]
\Require Spatial discretization \eqref{eqn:spatial_dis}, velocity discretization \eqref{eqn:velocity_dis}, and initial condition represented as SVD form $\mathbf{F}^{0,C} = U^0 \Sigma^0 (V^0)^\mathrm{T}$ (default SVD tolerance $\varepsilon_s=10^{-3}$). {$\Phi^{n+1/2}_i$ represents the potential approximation at $(x_i,(n+1/2)\Delta t)$.} 
\While{$t < T$}
    \State \textbf{Step 1: Semi-Lagrangian with ACA-SVD.} Use~\eqref{eqn:SL_matrix} as function handle; apply ACA-SVD to obtain $\mathbf{F}^{*,n+1/2}$.
    \State \textbf{Step 2: Solve $\Phi^{n+1/2}$ from the Poisson.} Compute density from $\mathbf{F}^{*,n+1/2}$ by summation; solve Poisson equation to obtain $\Phi^{n+1/2}$. 
    \State \textbf{Step 3: Fourier Update with Modified ACA-SVD.} 
    \State Apply FFT column-wise to $V^{*,n+1/2}$ to obtain $\tilde{\mathbf{F}}^{*,n+1/2}$. 
    \State Apply ACA-SVD to \eqref{eqn:FFT_matrix} with symmetric column pairing \eqref{eq:symmetric_column_pair} to obtain $\tilde{\mathbf{F}}^{*,n+1} = U^{*,n+1} \Sigma^{*,n+1} \left(\tilde{V}^{*,n+1} \right)^{\mathrm{T}}$.
    \State {Apply inverse FFT column-wise to the right singular factor $\tilde{V}^{*,n+1}$ to recover real-valued $\mathbf{F}^{*,n+1}$.}
    \State \textbf{Step 4: Second Semi-Lagrangian with ACA-SVD.} 
    \State Repeat Step 1 starting from $\mathbf{F}^{*,n+1}$ to obtain $\mathbf{F}^{n+1} = U^{n+1} \Sigma^{n+1} (V^{n+1})^\mathrm{T}$.
    \State Compute total mass $L^{n+1,*}$ of  $\mathbf{F}^{n+1}$ as defined in~\eqref{eq:total_mass_low_rank}.

    \State \textbf{Step 5: Density Correction.} Apply Lagrange correction:
    \[
    \mathbf{F}^{n+1,C} = U^{n+1} \left( \Sigma^{n+1} \cdot \frac{L_x}{L^{n+1,*}} \right) (V^{n+1})^\mathrm{T}.
    \]
    \State $t \gets t + \Delta t$
\EndWhile
\end{algorithmic}
\end{algorithm}

Goreinov and Tyrtyshnikov showed in their 2001 paper~\cite{goreinov2001maximal} that the CUR approximation converges when the “best” rows and columns are selected, as defined in the assumption of Theorem~\ref{thm:ACA}. 

{\begin{theorem}\label{thm:ACA}
For a CUR decomposition \( \mathcal{C}\mathcal{U}\mathcal{R} \) of a matrix \( A \), if \( \mathcal{U}^{-1} \) has the largest volume among all submatrices (i.e., has the largest absolute determinant), then
\[
\|A - \mathcal{C}\mathcal{U}\mathcal{R}\|_{\max} \le (k+1) \, \sigma_{k+1}(A),
\]
where \( \sigma_{k+1}(A) \) is the \((k{+}1)\)-th largest singular value of \( A \), and \( \| \cdot \|_{\max} \) the entrywise maximum norm.
\end{theorem}}
{However, the assumptions of Theorem \ref{thm:ACA} do not always hold with the proposed ACA algorithm with greedy search along columns and rows.} Nonetheless, we find the approximation to be effective and robust in practice when solving the Wigner-Poisson system, with a significantly reduced computational complexity compared with the corresponding full rank solver. We will defer a more detailed error analysis to future work, where we aim to understand why ACA, when paired with SVD, performs robustly in this class of nonlocal integral-differential equations.

\subsubsection{Structure-preserving property}\label{subsec:structure-preserving}
In Algorithm~\ref{alg:LR-WP solver}, we introduced a {modified} ACA procedure for the Fourier update step, to ensure the preservation of column-wise conjugate symmetry. In this subsection, we rigorously establish {such structure preservation}, formalized in the following theorem.
\begin{theorem}\label{thm:structure-preserving}
Let $A \in \mathbb{C}^{N_x \times N_v}$ be a special column-wise conjugate symmetric matrix satisfying
\[
A(:,j) = \overline{A(:,j_{opp})} \quad \text{for all } j = 1, \ldots, N_v,
\]
with $j_{opp}$ defined in \eqref{eq:symmetric_column_pair}. 
Then, the modified CUR decomposition of $A$ denoted by $\tilde{A}$ retains the same kind of column-wise conjugate symmetry.
\end{theorem}

\begin{proof}
{Recall ACA-SVD scheme in Section \ref{sec:ACA-SVD}.} Assume that the index sets $\mathcal{I}$ and $\mathcal{J}$ are still used to represent the selected the rows and columns, respectively, and $\mathcal{J}$ includes each index along with its conjugate symmetric counterpart, i.e. $\mathcal{J} = \mathcal{J}_{opp}$ as we do in the Algorithm \ref{alg:LR-WP solver}. Without loss of generality, assume the indices in $\mathcal{J}$ are ordered increasingly, and reorder $\mathcal{I}$ accordingly. {Let $\tilde{A} = \mathcal{C}\mathcal{U}\mathcal{R}$ be the form obtained from the modified CUR decomposition of $A$. }, where
\[
\mathcal{C} \in \mathbb{C}^{N \times r}, \quad \mathcal{U} \in \mathbb{C}^{r \times r}, \quad \mathcal{R} \in \mathbb{C}^{r \times N},
\]
{and $\mathcal{U}^{-1}$ is the submatrix $A(\mathcal{I},\mathcal{J})$.}

With this ordering and symmetric selection, the matrices $\mathcal{C}$, $\mathcal{U}$, and $\mathcal{R}$ satisfy:
\begin{equation} \label{eq:symmetry}
\mathcal{C}(:,k) = \overline{\mathcal{C}(:,k_{opp,r})}, \quad \mathcal{U}^{-1}(:,k) = \overline{\mathcal{U}^{-1}(:,k_{opp,r})}, \quad \mathcal{R}(:,j) = \overline{\mathcal{R}(:,j_{opp})},
\end{equation}
where 
\begin{equation*}
k_{opp,r} = 
    \begin{cases}
        \mod(r+1-k,r)+1, \, \text{if } 1 \in \mathcal{J} \\
        \quad r+1 - k, \textbf{else}.
    \end{cases}
\end{equation*}
Define the reversal matrix $J \in \mathbb{C}^{r \times r}$ by 
$J(i,k) = \delta_{i,k_{opp,r}}$ which is form of
\[
J = 
\begin{bmatrix}
0 & 0 & \cdots & 0 & 1 \\
0 & 0 & \cdots & 1 & 0 \\
\vdots & \vdots & \reflectbox{$\ddots$} & \vdots & \vdots \\
0 & 1 & \cdots & 0 & 0 \\
1 & 0 & \cdots & 0 & 0
\end{bmatrix} \, \text{if } 1 \in \mathcal{J}, \, \text{or else }
J = 
\begin{bmatrix}
1 & 0 & \cdots & 0 & 0 \\
0 & 0 & \cdots & 0 & 1 \\
0 & 0 & \cdots & 1 & 0 \\
\vdots & \vdots & \reflectbox{$\ddots$} & \vdots & \vdots \\
0 & 1 & \cdots & 0 & 0 
\end{bmatrix},
\]
and both cases, we have the identity $J = J^{-1} = J^{\mathrm{T}}$. 
Then, from the symmetry of $\mathcal{U}^{-1}$, we can write $\mathcal{U}^{-1} = \overline{\mathcal{U}^{-1}} J$. Using the identity $J = J^{-1} = J^{\mathrm{T}}$, it follows that
\[
\mathcal{U} = (\overline{\mathcal{U}^{-1}} J)^{-1} = J^{-1} (\overline{\mathcal{U}^{-1}})^{-1} = J (\overline{\mathcal{U}}),
\]
which implies for all $k$,
\[
\mathcal{U}(k,:) = \overline{\mathcal{U}(k_{opp,r},:)}.
\]

We now verify the conjugate symmetry of $\tilde{A} = \mathcal{C}\mathcal{U}\mathcal{R}$: $\forall \, i,j \in \{1,...,N\}$, 
\begin{align*}
\tilde{A}(i,j) &= \sum_{k=1}^r \sum_{\ell=1}^r \mathcal{C}(i,\ell)\, \mathcal{U}(\ell,k)\, \mathcal{R}(k,j) = \sum_{k=1}^r \sum_{\ell=1}^r \overline{\mathcal{C}(i, \ell_{opp,r})}\, \overline{\mathcal{U}(\ell_{opp,r}, k)}\, \overline{\mathcal{R}(k, j_{opp})} \\
&= \overline{ \sum_{k=1}^r \sum_{\ell_{opp,r}=1}^r \mathcal{C}(i, \ell_{opp,r})\, \mathcal{U}(\ell_{opp,r}, k)\, \mathcal{R}(k, j_{opp})} = \overline{ \sum_{k=1}^r \sum_{\ell=1}^r \mathcal{C}(i, \ell)\, \mathcal{U}(\ell, k)\, \mathcal{R}(k, j_{opp})} = \overline{\tilde{A}(i, j_{opp})}.
\end{align*}

Therefore, $\tilde{A}$ satisfies the same column-wise conjugate symmetry as $A$, which shows that the modified CUR decomposition preserves this structure.
\end{proof}

\section{Numerical simulations}\label{sec:results}
In this section, we present two numerical tests to {showcase} our new adaptive-rank Wigner–Poisson solver: two-stream instability and strong Landau damping. Unless otherwise specified, we use the same tolerance settings of $\varepsilon_c = 10^{-4}$, $\varepsilon_s = 10^{-3}$ as in \cite{zheng2025semi}, with no maximum rank limitation.

\subsection{Two-Stream Instability}
We test the two-stream instability problem using the same setup as in equation \eqref{eqn:ini_TSI} for the full-rank solver discussed in Section \ref{subsec:rank analysis}, but now with the proposed adaptive-rank Wigner–Poisson solver. In Figure \ref{fig:TSI_lowrank_H_T45}, we present results for four different values of the quantum parameter $H = 0.1, 0.5, 1, 8$, using the same numerical settings as in the full-rank simulations: $N_x = N_v = 512$, CFL = 50, and final time $T = 45$. As shown in Figure \ref{fig:TSI_lowrank_H_T45}, the results for $H = 1, 8$ are visually indistinguishable from the corresponding full-rank solutions in Figures \ref{fig:fullrank_T45_1} and \ref{fig:fullrank_T45_2}. For $H = 0.1, 0.5$, the low rank solutions {obtained by the new adaptive-rank Wigner–Poisson solver} {well} capture the main structure of the full-rank results. As shown in Figure \ref{fig:TSI_mass_momentum_H}, mass is conserved up to machine precision, and momentum is conserved within $10^{-4}$. Note that this is capped at {the truncation error of the scheme} over the 50 plasma periods we compute. 
To verify the structure-preserving property, we present Figure \ref{fig:TSIimaginary}. As shown in the figure, the integral of the imaginary part after the Fourier update step is zero up to machine precision, which is consistent with theoretical expectations. {In addition, we perform a convergence test to verify the temporal accuracy of the proposed solver. As shown in Figure \ref{fig:TSI_convergence}, the log-log plots of the $L^2-$error versus time step size $\Delta t$ for different values of $H = 0.5, 1, 8$ exhibit slopes consistent with second-order accuracy before the error reaches the level of $10^{-3}$, as indicated by the reference line of slope 2. Once the error approaches $10^{-3}$, an order reduction is observed, which is attributed to the prescribed tolerance settings ($\varepsilon_c$ = $10^{-4}$ and $\varepsilon_s$ = $10^{-3}$). The simulations are performed up to final time $T = 50$ using a mesh of $512 \times 512$ grid points.}

\begin{figure}[htbp]
\centering
\begin{subfigure}{0.49\textwidth}
\centering
\includegraphics[width=\linewidth]{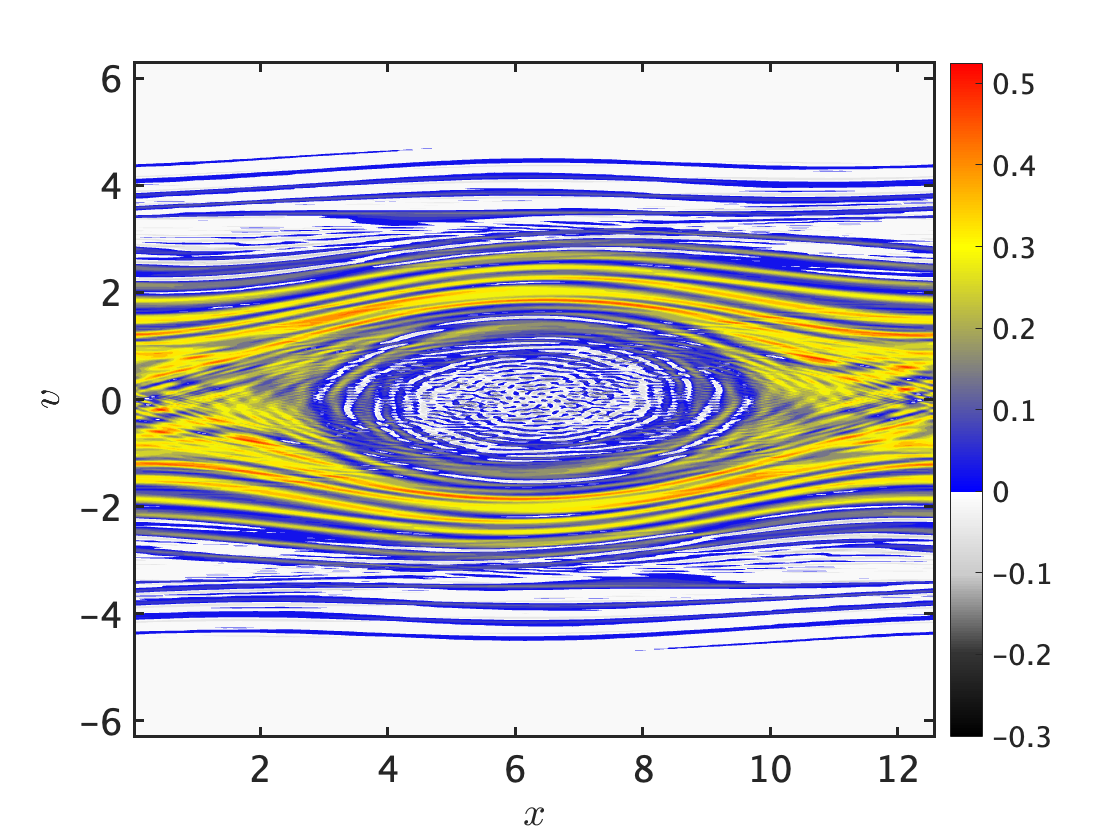}
\caption{$H = 0.1$}
\label{fig:TSI_lowrank_H01_T45}
\end{subfigure}
\begin{subfigure}{0.49\textwidth}
\centering
\includegraphics[width=\linewidth]{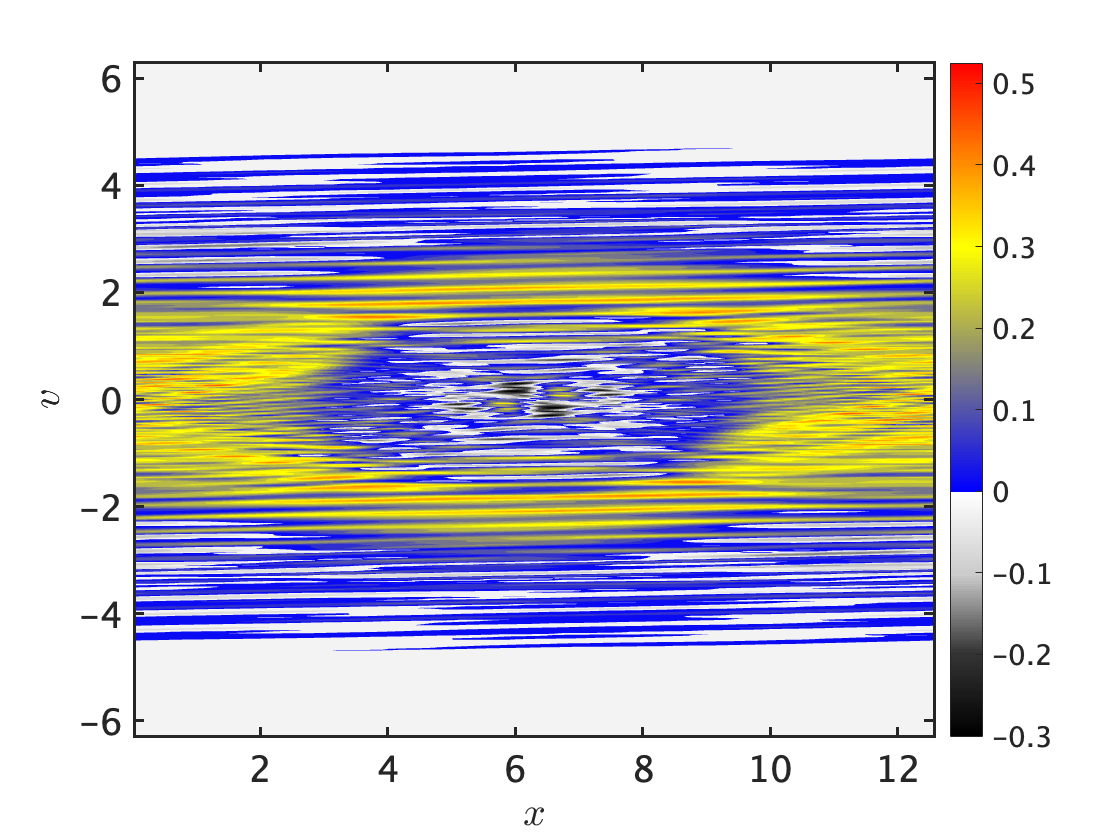}
\caption{$H = 0.5$}
\label{fig:TSI_lowrank_H05_T45}
\end{subfigure}
\begin{subfigure}{0.49\textwidth}
\centering
\includegraphics[width=\linewidth]{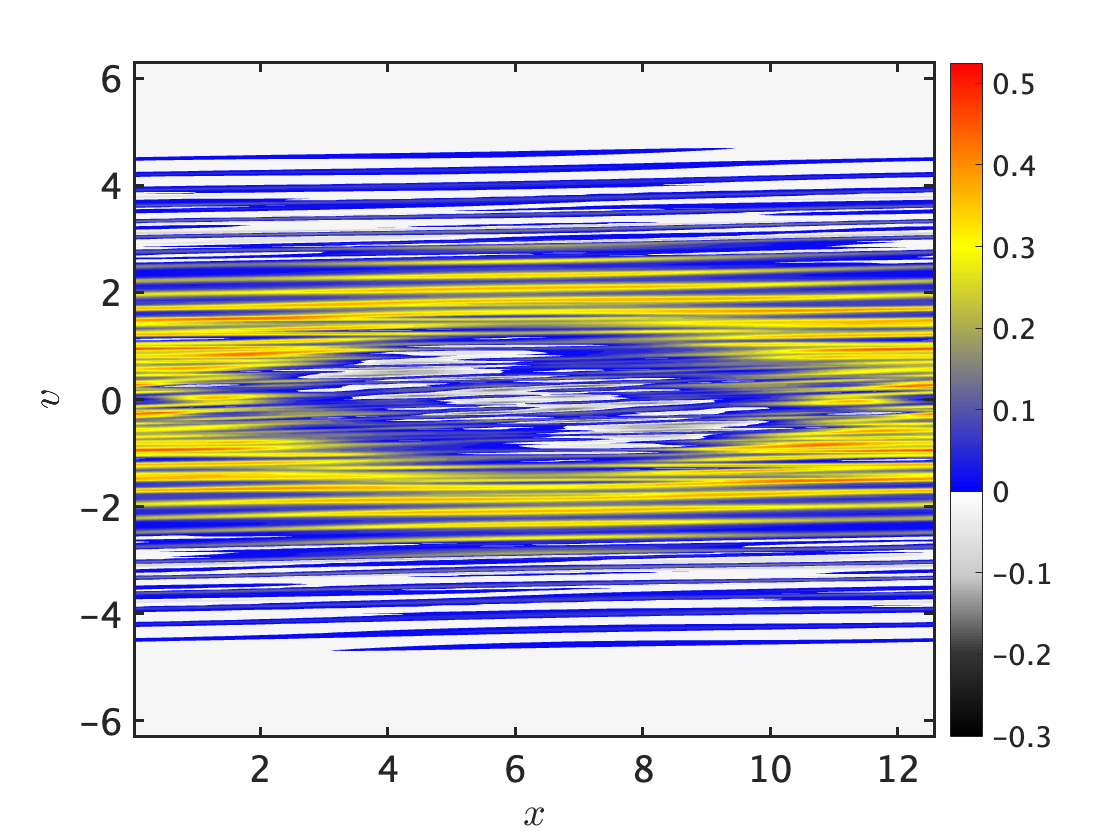}
\caption{$H = 1$}
\label{fig:TSI_lowrank_H1_T45}
\end{subfigure}
\begin{subfigure}{0.49\textwidth}
\centering
\includegraphics[width=\linewidth]{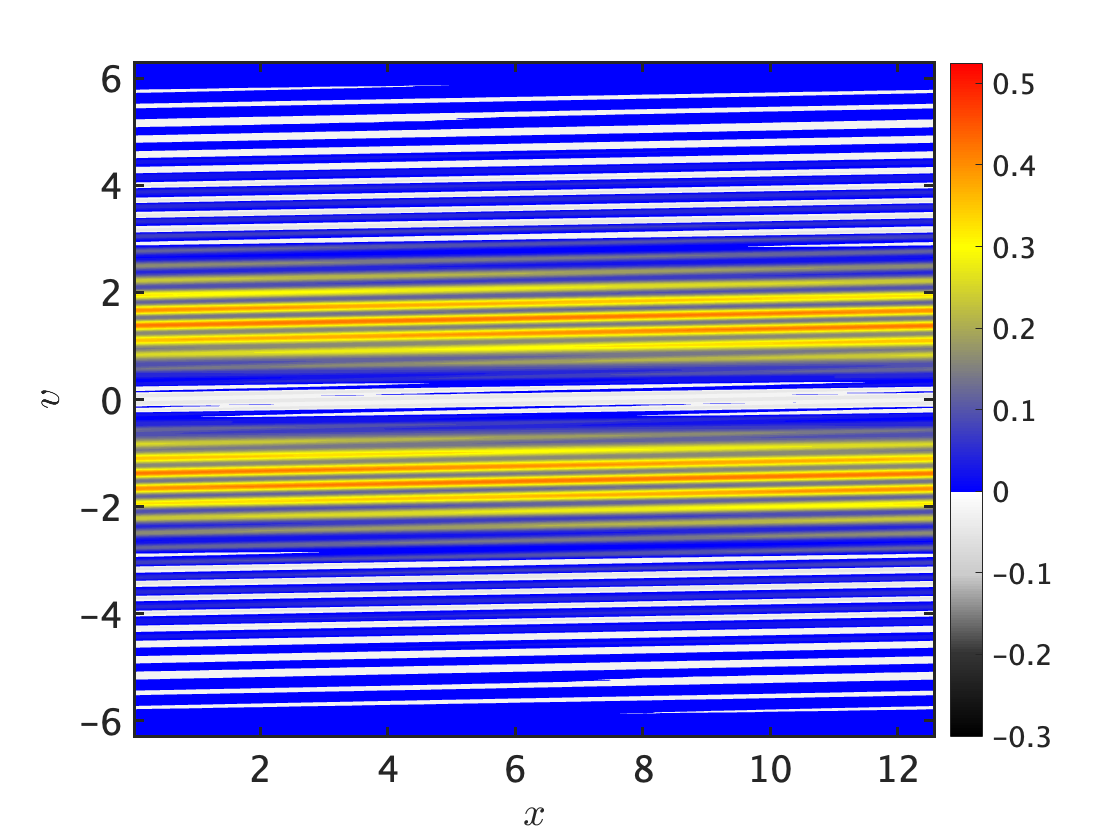}
\caption{$H = 8$}
\label{fig:TSI_lowrank_H8_T45}
\end{subfigure}   
\caption{Phase space solutions of Wigner–Poisson (two-stream instability): This figure shows the phase space solutions of the {low rank} simulations for different $H$ at $T = 45$. Compared with the full-rank simulations, the adaptive rank solver can accurately solve the equations, especially for $H \ge 1$. By increasing the mesh size and solver tolerance, more details can be captured, especially when $H$ is close to zero.}
\label{fig:TSI_lowrank_H_T45}
\end{figure}

\begin{figure}[htbp]
\centering
\begin{subfigure}{0.49\textwidth}
\centering
\includegraphics[width=\textwidth]{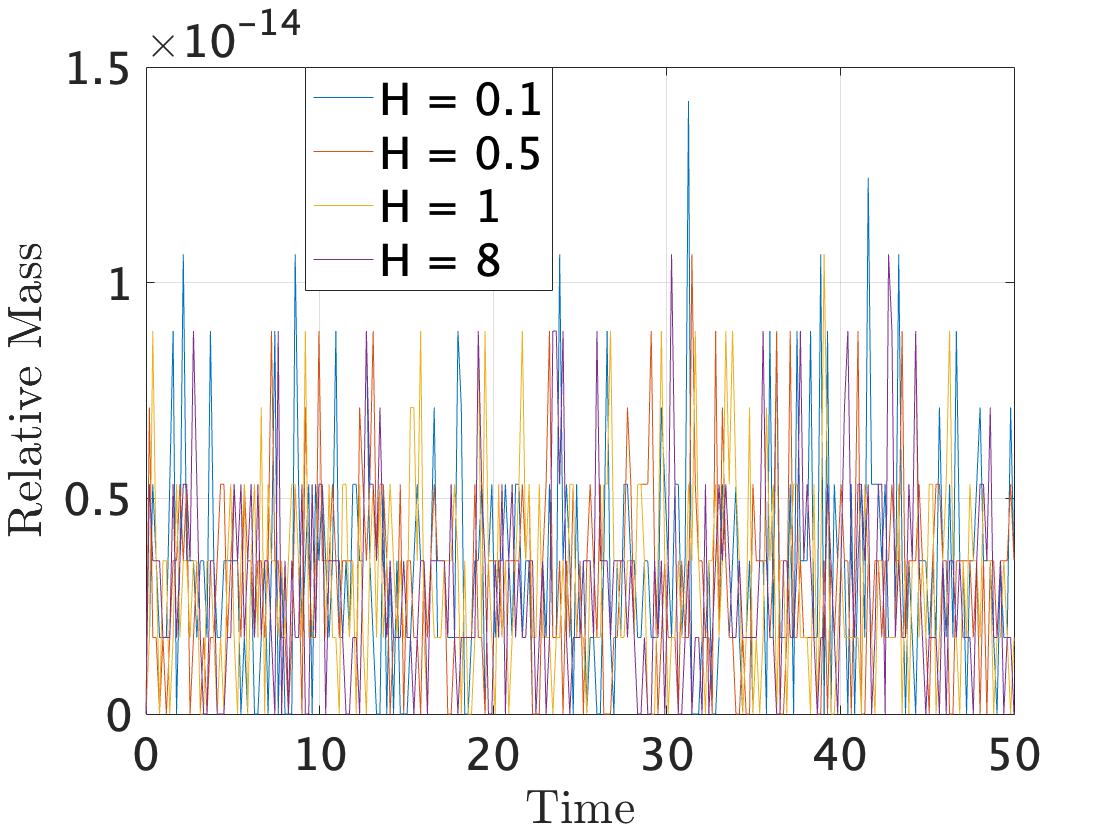}
\caption{Relative mass errors for all $H$.}
\label{fig:TSI_mass_H_all}
\end{subfigure}
\begin{subfigure}{0.49\textwidth}
\centering
\includegraphics[width=\linewidth]{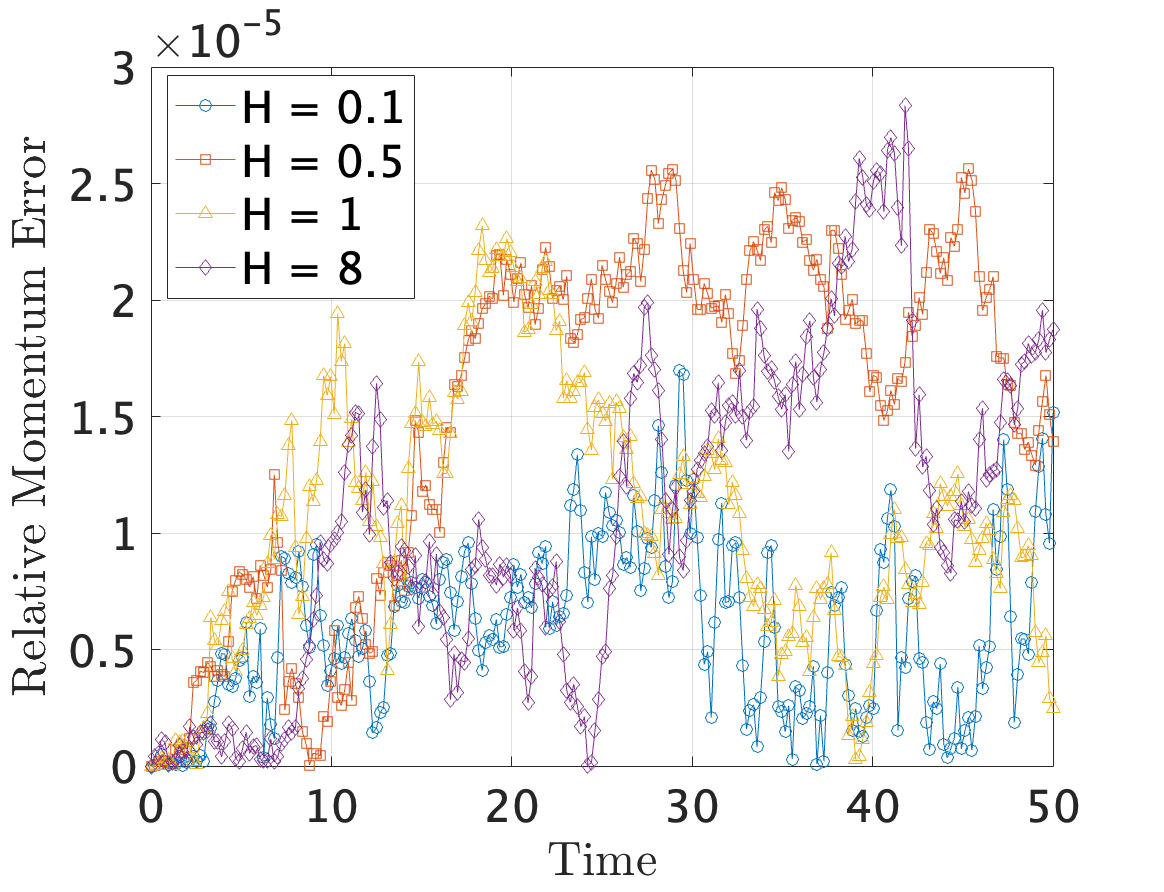}
\caption{Relative momentum errors for all $H$.}
\label{fig:TSI_momentum_H_all}
\end{subfigure}
\caption{Conservation law analysis (two-stream instability): This figure illustrates the evolution of relative mass and momentum over time (up to $T = 50$), using the adaptive-rank Wigner–Poisson solver. Results are shown for different values of the quantum parameter $H$ with a fixed spatial resolution of $512 \times 512$. Mass is conserved up to machine precision in all cases, while momentum error remains controlled within $10^{-4}$.}
\label{fig:TSI_mass_momentum_H}
\end{figure}



We also analyze the computational complexity of our solver. With fixed $\Delta t = 0.1$, final time $T = 50$, and number of grid points $N_x = N_v$, we observe that our adaptive-rank Wigner–Poisson solver exhibits approximately linear growth with respect to the number of grid points. See Figure \ref{fig:TSIrunningtime} for details. This is promising, as we hope to develop an adaptive rank solver for the 6D non-local Wigner–Poisson system, and these results suggest this approach is effective. Here, we present results for $H = 0.5, 1, 8$. For $H = 0.1$, as discussed in Section \ref{subsec:rank analysis}, the system is close to the Vlasov–Poisson regime and the solution's rank approaches full rank. Thus, a much finer grid is needed before the scaling becomes $\mathcal{O}(N)$. {For comparison, we also include the running time of the full-rank solver in the same plot, up to $N_x = 512$ due to its high computational cost. The observed quadratic growth agrees with the expected scaling of $\mathcal{O}(N^2)$.}

\begin{figure}[htbp]
    \centering
    
\begin{subfigure}{0.49\textwidth}
        \centering
        \includegraphics[width=\linewidth]{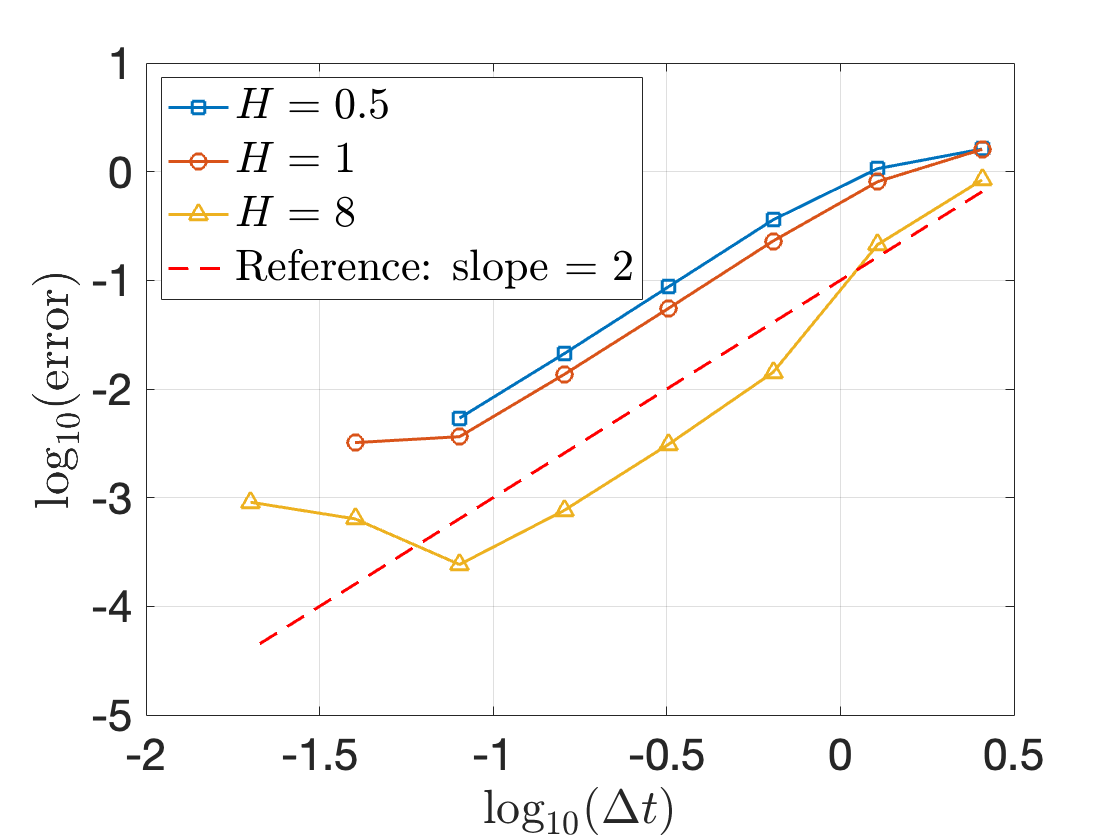}
        \caption{Convergence analysis}
        \label{fig:TSI_convergence}
    \end{subfigure}
    \hfill
    \begin{subfigure}{0.49\textwidth}
        \centering
        \includegraphics[width=\linewidth]{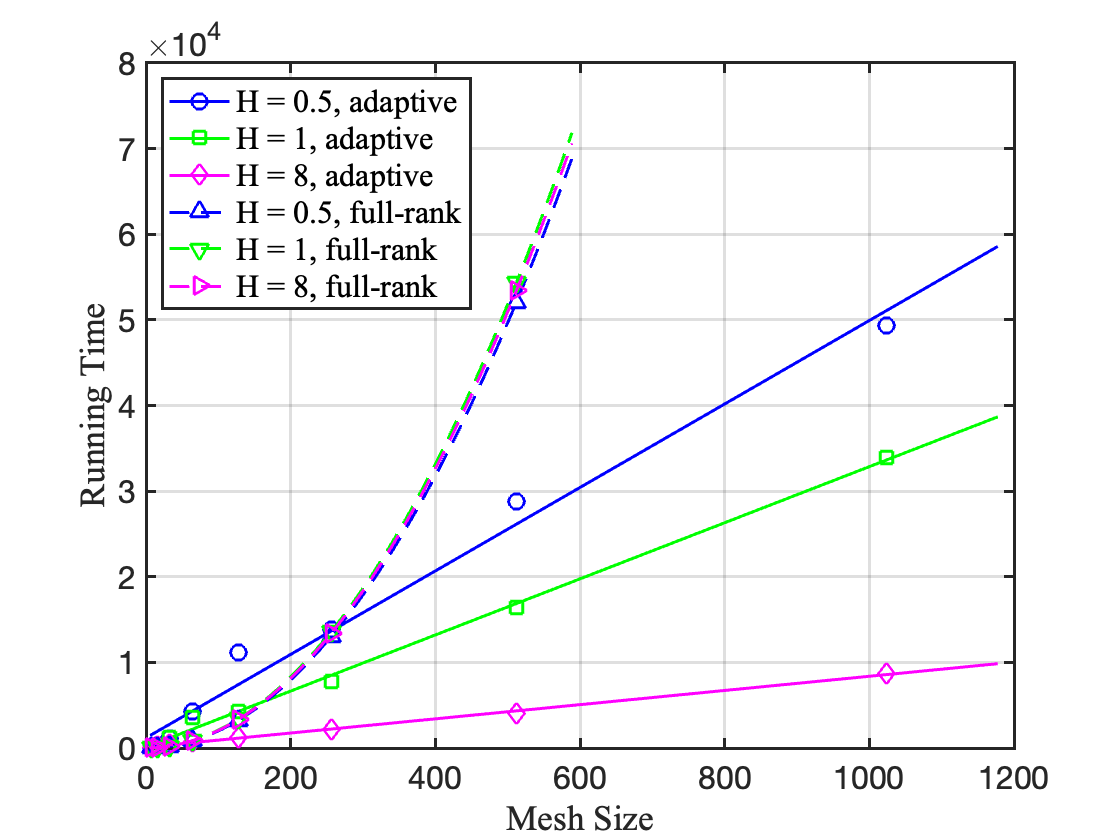}
        \caption{Computational cost analysis}
        \label{fig:TSIrunningtime}
    \end{subfigure}
    \caption{
        Two numerical studies for two-stream instability are conducted for quantum parameters \( H = 0.5, 1, 8 \) up to final time \( T = 50 \). 
        \textbf{(a)} Convergence analysis: Log-log plot of the error versus time step size \( \Delta t \) with a mesh grid of $512 \times 512$. A reference line of slope 2 is shown, confirming second-order accuracy. 
        \textbf{(b)} Computational cost analysis: Running time as a function of the number of grid points per dimension with fixed time step size \( \Delta t = 0.1 \). 
    }
    \label{fig:TSI_study}
\end{figure}

\subsection{Strong Landau damping}
Next we present our results for the strong Landau damping. Landau damping is the collisionless damping of waves in a plasma described by the Vlasov equation. Due to the wave-particle interaction, the electric field interacting with electrons, energy is transferred from the lowest Fourier mode of the electric field to the particles, and subsequently from the particles to higher Fourier modes of the electric field, resulting in damping of the lowest mode over time. Landau damping is a linearized result and only applies over short timescales. When the perturbation is strong enough, the electrons eventually become trapped and form a structure which sustains energy in the low modes of the electric field. Hence, strong Landau damping behaves differently due to electron trapping. However, when quantum effects such as tunneling are present, and if tunneling is strong enough, electrons are no longer trapped and the field will still damp as in the case of weak Landau damping.

Here, we run the strong Landau damping problem on the domain $x \in [0, 5\pi]$, $v \in [-2\pi, 2\pi]$, with the initial condition:
\begin{equation}\label{eqn:ini_StrongLD}
f_0(x,v) = \frac{e^{-v^2/2}}{\sqrt{2\pi}}(1+0.2\cos(0.4x)).
\end{equation}

\begin{figure}[htbp]
\centering
\begin{subfigure}{0.49\textwidth}
\centering
\includegraphics[width=\linewidth]{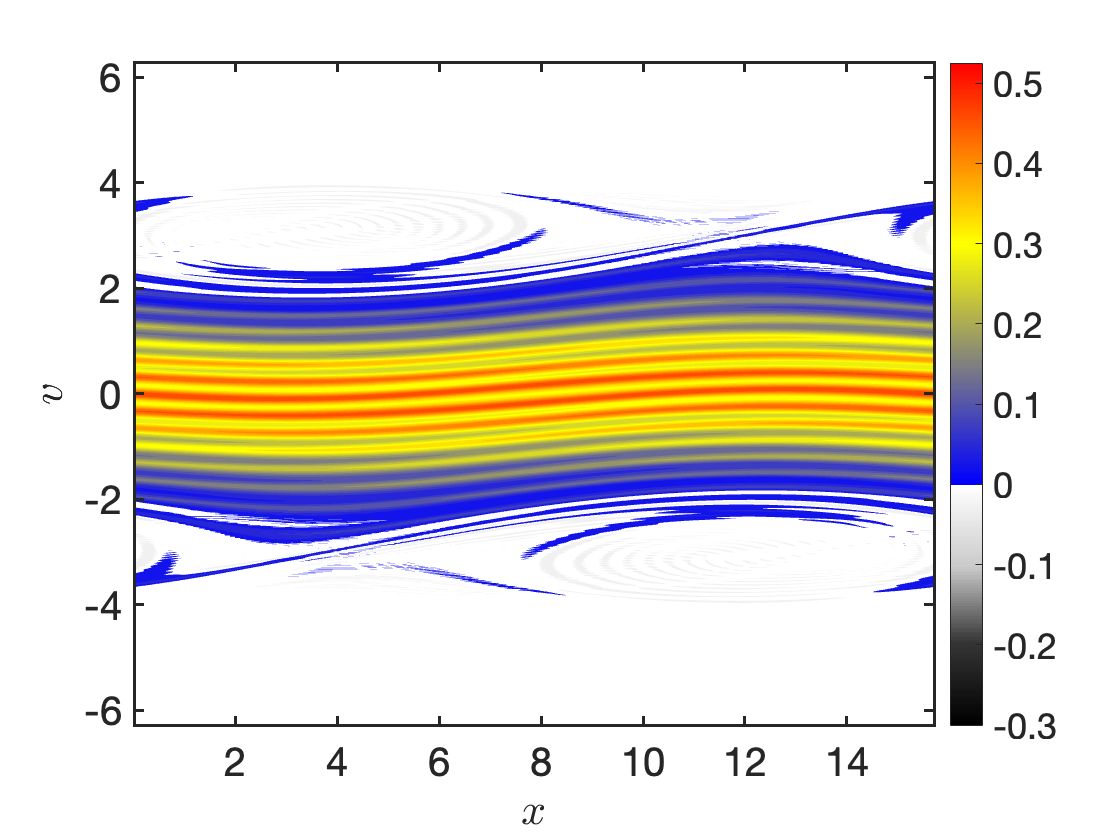}
\caption{$H = 0.1$}
\label{fig:SLD_fullrank_H01_T45}
\end{subfigure}
\begin{subfigure}{0.49\textwidth}
\centering
\includegraphics[width=\linewidth]{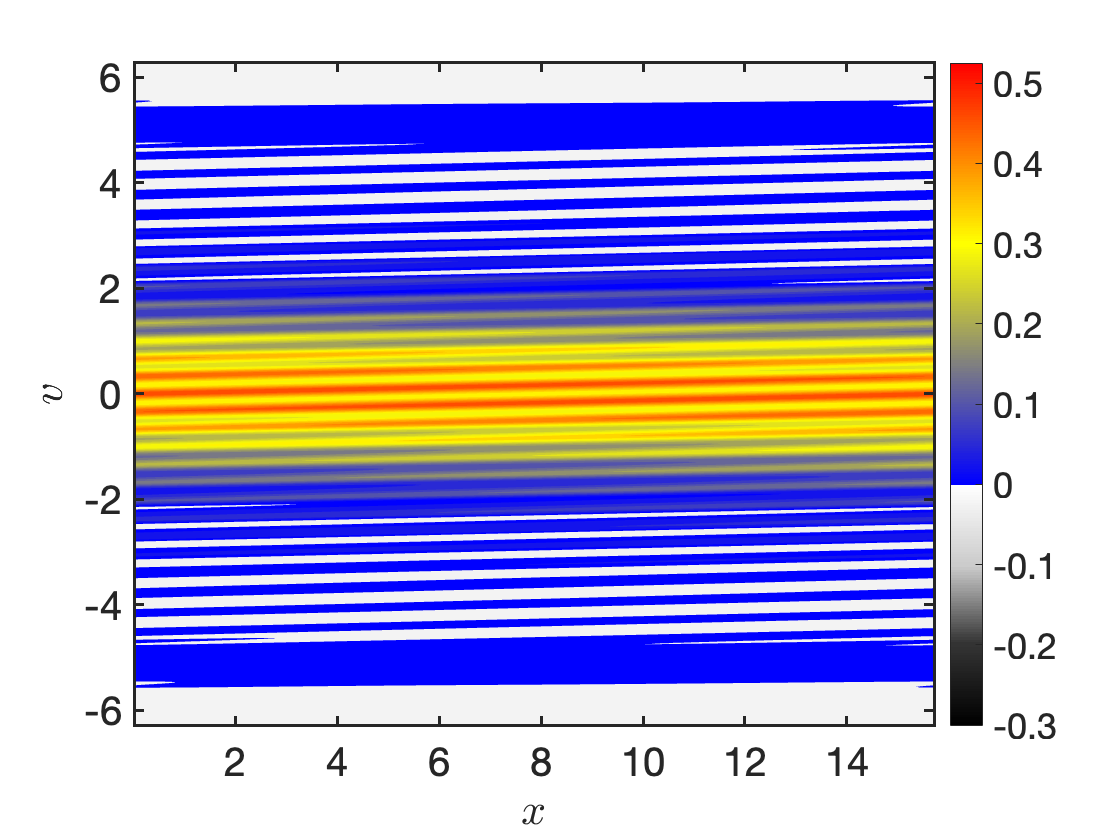}
\caption{$H = 8$}
\label{fig:SLD_fullrank_H8_T45}
\end{subfigure}    
\caption{Phase space solutions of Wigner–Poisson (strong Landau damping): This figure shows the phase space solutions of the full-rank simulations for $H = 0.1, 8$ at $T = 45$ with a resolution of $1024 \times 1024$. These results serve as a benchmark for comparison with {low rank} simulations.}
\label{fig:SLD_fullrank_H_T45}
\end{figure}

In Figure \ref{fig:SLD_fullrank_H_T45}, we present the full-rank solutions of this example, and in Figure \ref{fig:SLD_lowrank_H_T45}, we show the low rank solutions obtained by our solver for {the representative} quantum parameter cases $H = 0.1, 8$, with $N_x = N_v = 512$, CFL = $50$, and $T = 45$. As seen in these figures, the low rank Landau damping solutions are visually indistinguishable from the full-rank solutions. As shown in Figure \ref{fig:SLD_mass_momentum_H}, {for $H = 0.1,0.5,1,8$}, mass is conserved up to machine precision, and momentum is conserved within $10^{-4}$. Again, we note that this is capped at the tolerance of the SVD truncation over the 50 plasma periods we simulate. The integration of the imaginary part is also zero up to machine precision (see Figure \ref{fig:SLDimaginary}), demonstrating that our algorithm preserves key structures in the solution of the Wigner–Poisson system.
\begin{figure}[htbp]
    \centering
    \begin{subfigure}{0.49\textwidth}
        \centering
        \includegraphics[width=\linewidth]{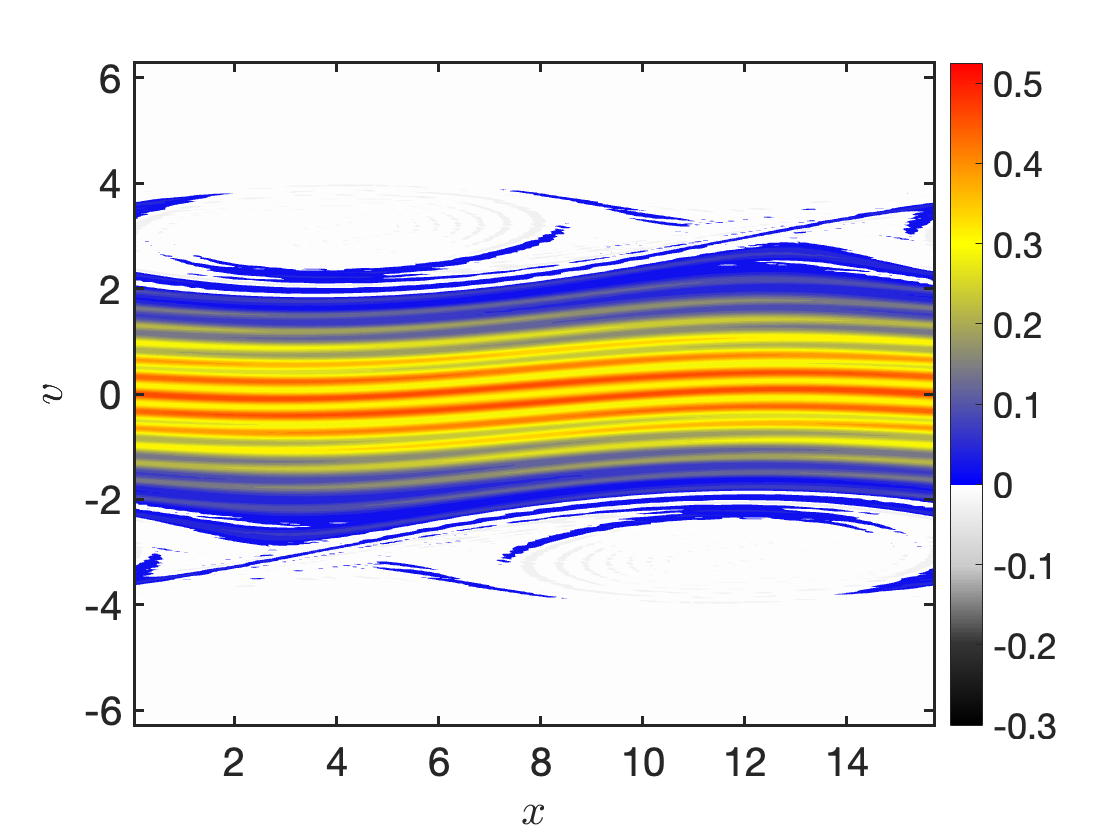}
        \caption{$H = 0.1$}
        \label{fig:SLD_lowrank_H01_T45}
    \end{subfigure}
    \centering
    \begin{subfigure}{0.49\textwidth}
        \centering
        \includegraphics[width=\linewidth]{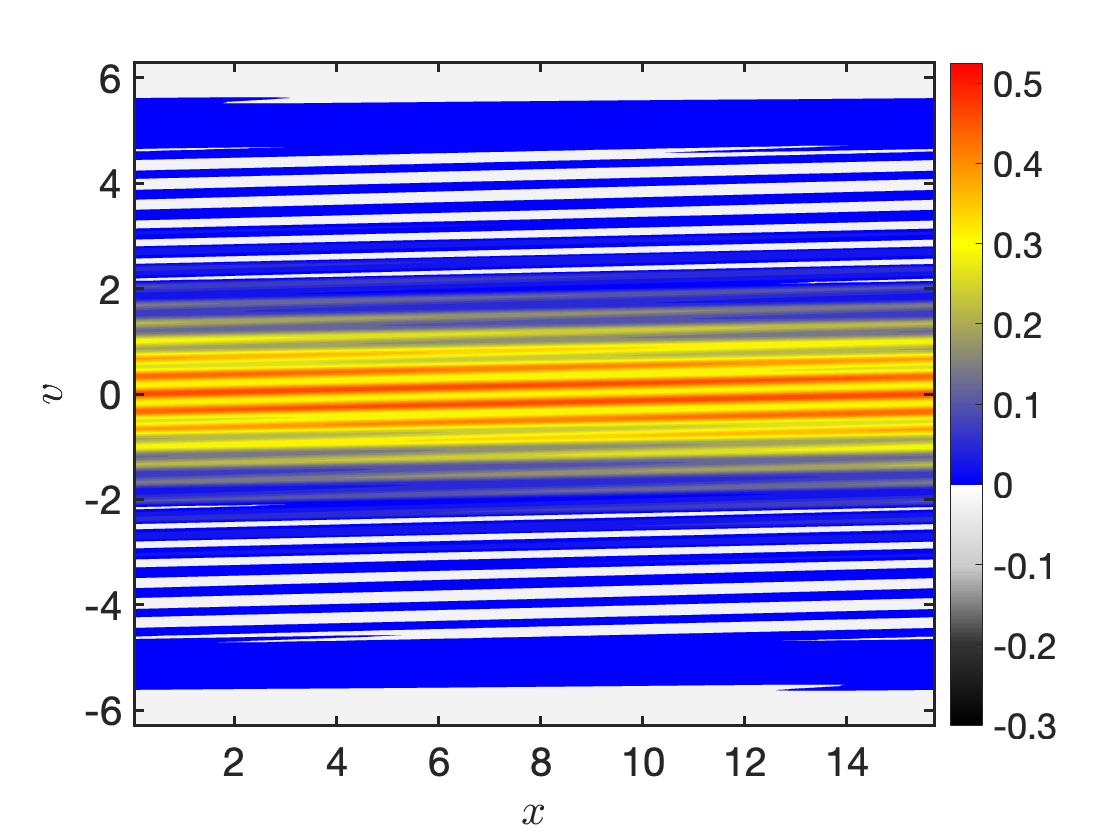}
        \caption{$H = 8$}
        \label{fig:SLD_lowrank_H8_T45}
    \end{subfigure}    
    \caption{Phase space solutions of Wigner Poisson (strong Landau damping): This figure shows the phase space solutions of the {low rank} simulations for $H = 0.1, 8$ at $T = 45$ at a resolution of $512 \times 512$. Compared with full rank results, the adaptive-rank solver is able to accurately approximate the solution.}
    \label{fig:SLD_lowrank_H_T45}
\end{figure}

\begin{figure}[htbp]
  \centering
  \begin{subfigure}{0.49\textwidth}
        \centering
        \includegraphics[width=\textwidth]{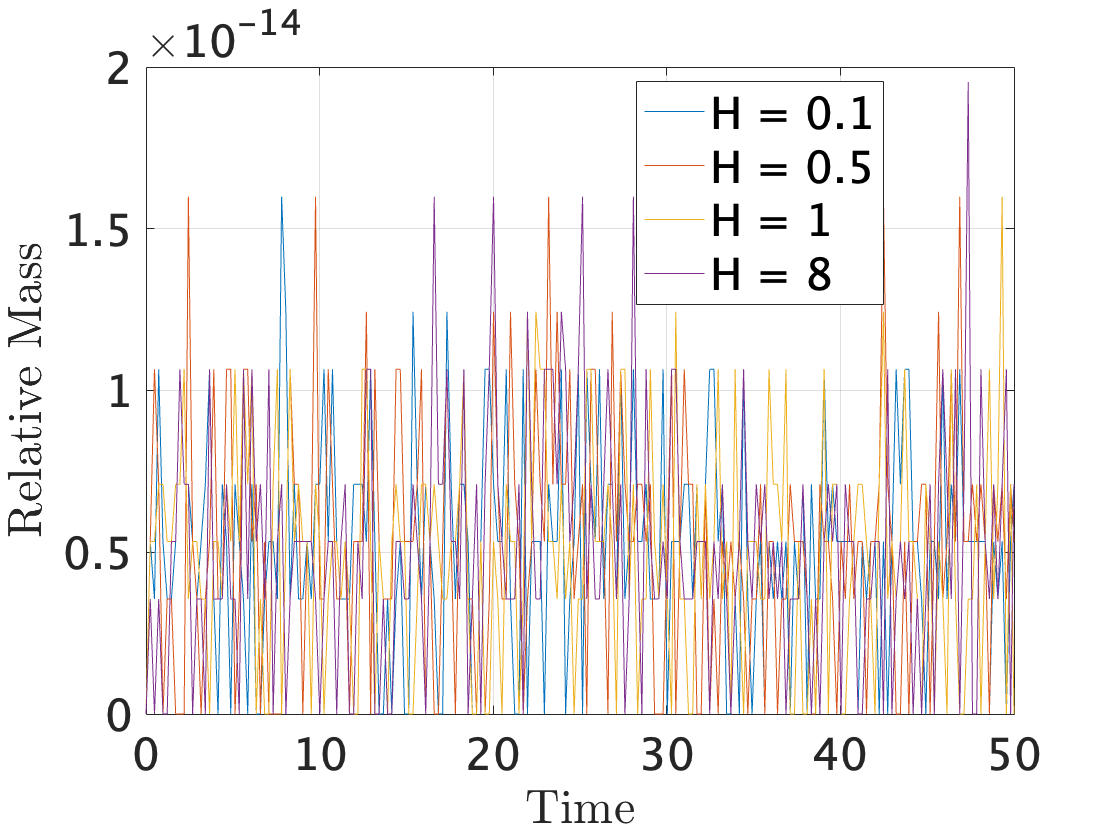}
        \caption{The relative mass errors for all $H$.}
        \label{fig:SLD_mass_H_all}
    \end{subfigure}
    \begin{subfigure}{0.49\textwidth}
        \centering
        \includegraphics[width=\linewidth]{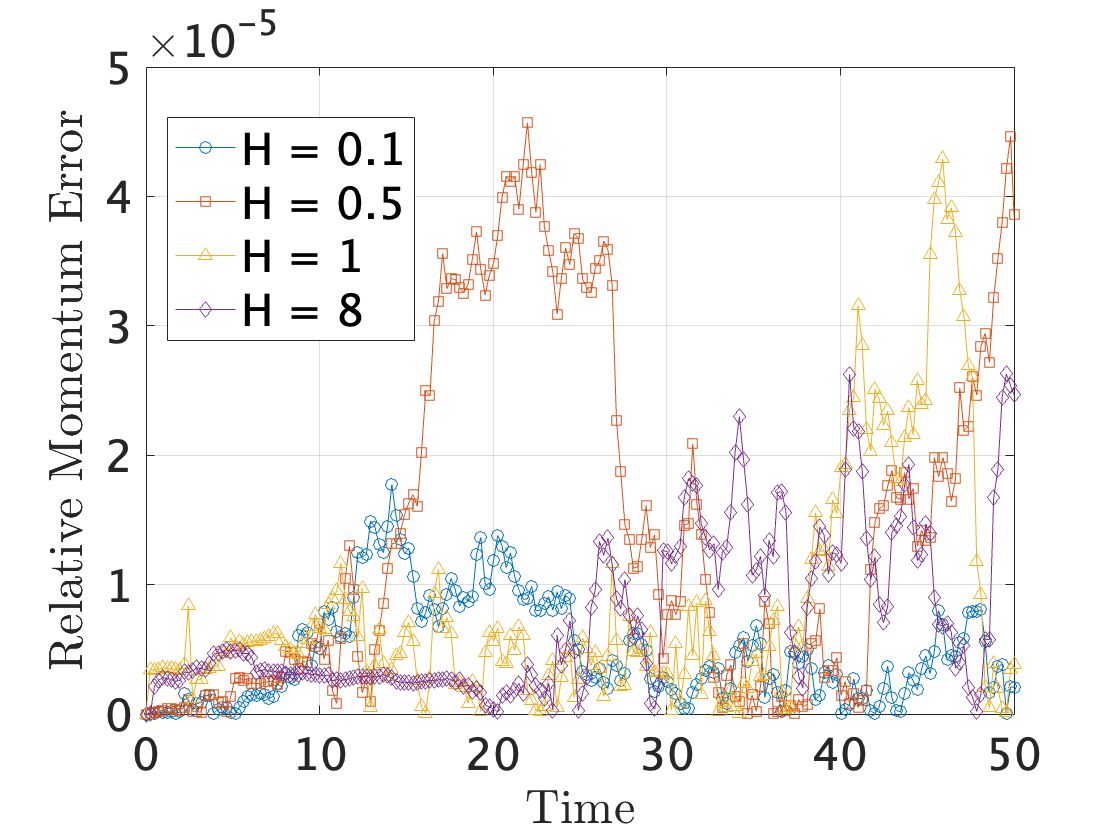}
         \caption{The relative momentum errors for all $H$.}
        \label{fig:SLD_momentum_H_all}
    \end{subfigure}
  \caption{Conservation law analysis (strong Landau damping): This figure illustrates the evolution of relative mass and momentum errors over time (up to $T = 50$), using the adaptive-rank Wigner–Poisson solver. Results are shown for different values of quantum parameter {$H = 0.1,0.5,1,8$} with a fixed spatial resolution of $512 \times 512$. Mass is conserved up to machine precision, and momentum error remains within $10^{-4}$ across all cases.}
  \label{fig:SLD_mass_momentum_H}
\end{figure}

\begin{figure}[htbp]
\centering
\begin{subfigure}{0.49\textwidth}
    \includegraphics[width=\textwidth]{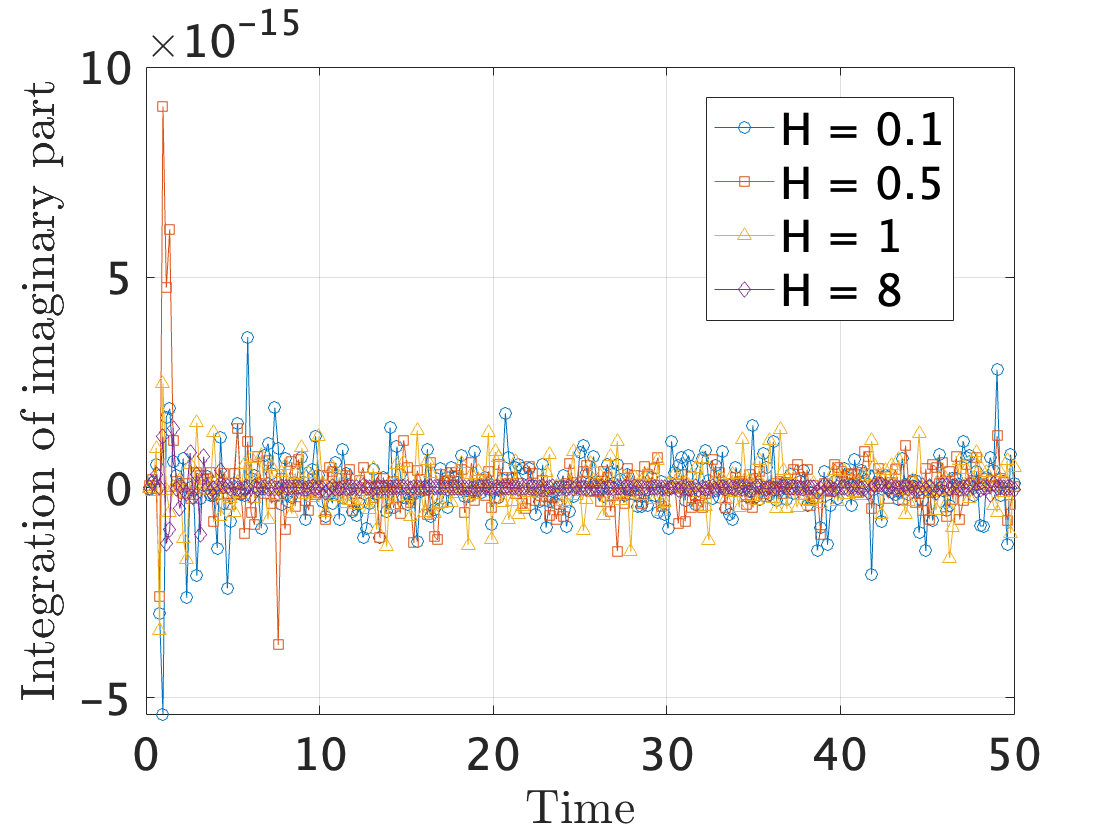}
\caption{Two-stream instability.}
\label{fig:TSIimaginary}
\end{subfigure}
\hfill
\begin{subfigure}{0.49\textwidth}
    \centering
  \includegraphics[width=\textwidth]{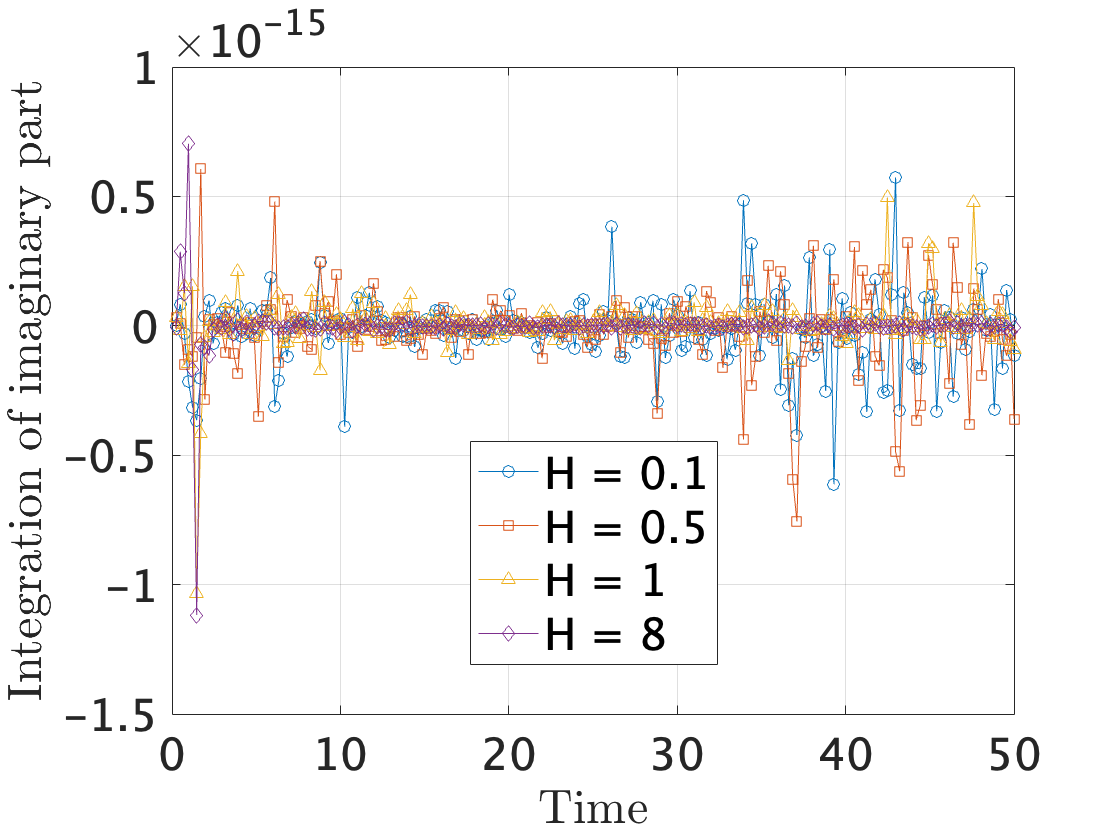}
  \caption{Strong Landau damping.}
  \label{fig:SLDimaginary}
\end{subfigure}
\caption{Structure-preserving analysis: These two figures illustrate the evolution of the integration of imaginary part after Fourier update over time (up to $T = 50$), using the adaptive-rank Wigner–Poisson solver. Results are shown for different values of quantum parameter {$H = 0.1,0.5,1,8$} with a fixed spatial resolution of $512 \times 512$. The integration is zero up to machine precision in all cases.}
\end{figure}

Furthermore, we present plots showing the electrostatic energy for strong Landau damping in the Wigner–Poisson system. In \cite{suh1991numerical}, the predicted damping rate for the Wigner case is $\gamma = 0.1516$. As seen in Figure \ref{fig:low_StrLan_rate_512_H8}, our simulation for $H = 8$ yields a damping rate of $\gamma = 0.1511$ for $N_x = 512$, which is in excellent agreement with the predicted value. 

\begin{figure}[htbp]
  \centering
    \begin{subfigure}{0.49\textwidth}
        \centering
        \includegraphics[width=\linewidth]{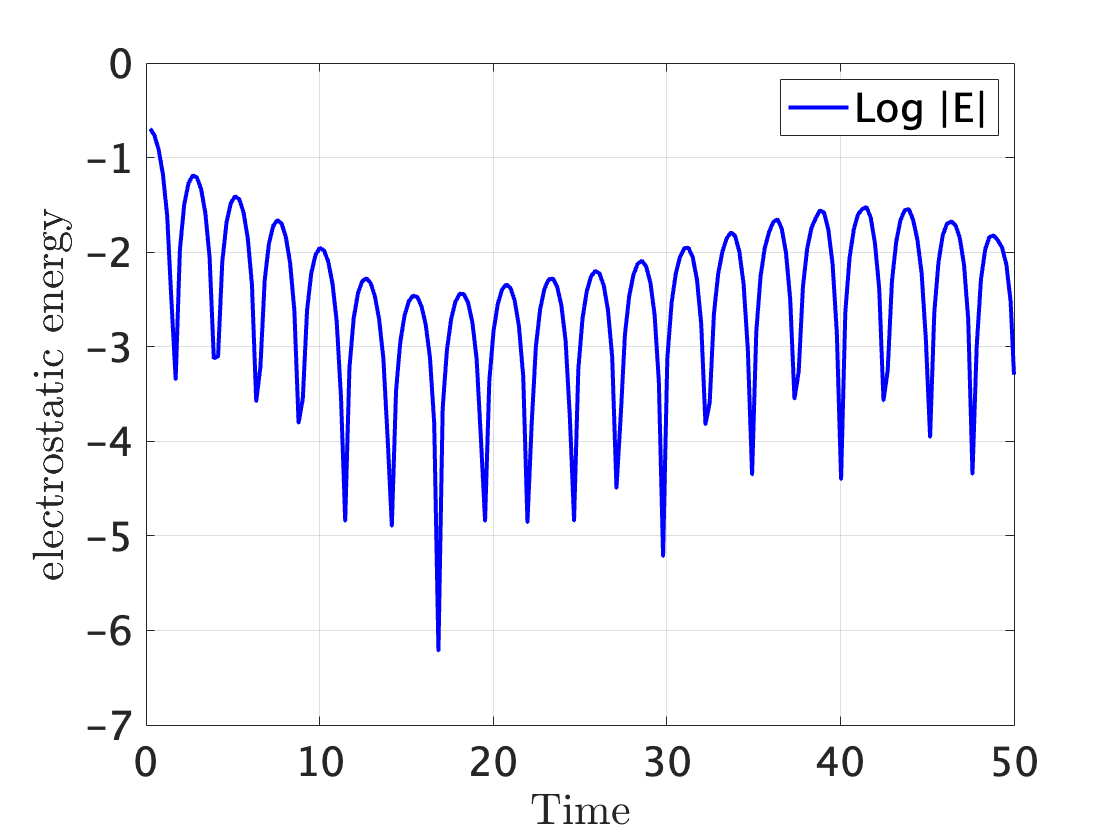}
        \caption{$H = 0.1$}
        \label{fig:low_StrLan_rate_512_H01}
    \end{subfigure}
    \begin{subfigure}{0.49\textwidth}
        \centering
        \includegraphics[width=\linewidth]{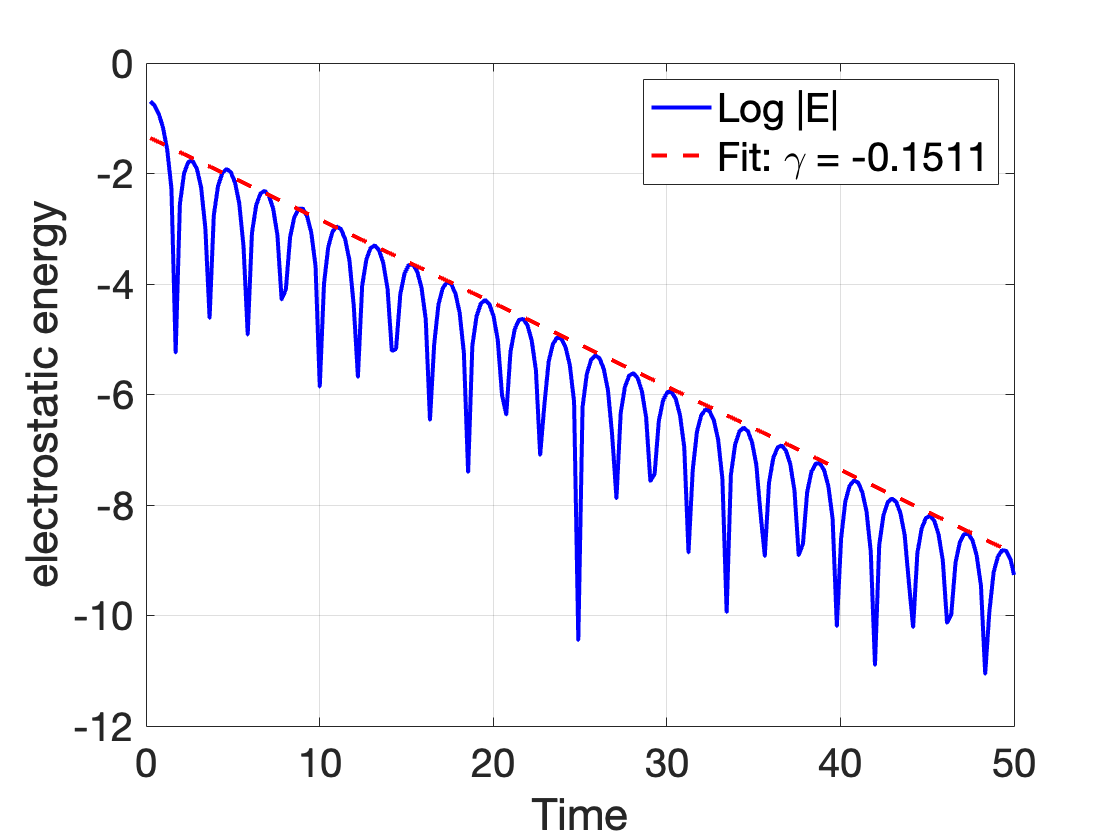}
        \caption{$H = 8$}
        \label{fig:low_StrLan_rate_512_H8}
    \end{subfigure}
    \caption{Evolution of electrostatic energy (strong Landau damping) in the Wigner–Poisson system with $N_x = 512$ and quantum parameters $H = 0.1$ and $H = 8$. The computed damping rates align closely with the theoretical prediction $\gamma = 0.1516$ from \cite{suh1991numerical}.}
  \label{fig:low_StrLan_rate}
\end{figure}

\section{Conclusions}

In this paper, we proposed a mass-conserving adaptive-rank Wigner–Poisson solver with linear computational cost. The method conserves momentum up to the truncation error used in the SVD between ACA steps. As discussed in Section~\ref{subsec:rank analysis}, the key observation motivating this work is that solutions to the Wigner–Poisson system exhibit a low rank structure. Unlike the Vlasov–Poisson system, Wigner–Poisson exhibits finite-size structures, allowing refinement in both space and time over long simulations. These insights were obtained from a {full-rank} method. The results from our proposed conservative adaptive-rank solver are strongly encouraging. They demonstrate that as the non-dimensional Planck’s constant $H$ enters the quantum regime ($H \ge 0.1$), low rank solutions adequately capture the system’s dynamics and enable long-time simulations with $\mathcal{O}(N)$ scaling. This opens the door to the possibility of 6D simulations of this nonlocal integral-differential equation, which is intractable on modern computing platforms using traditional parallel algorithms due to the extremely high communication cost caused by the model’s nonlocality.

In future work, we intend to explore other conservation strategies based on local approximations, extend the algorithm to a {high order} tensor framework for 6D Wigner–Poisson simulations, and develop collision operators to account for Coulomb interactions. Ultimately, we aim to pair this algorithm with simulations of $\alpha$-particles to study the stopping power problem relevant to the National Ignition Facility.

\section*{Acknowledgements}

The authors acknowledge support from AFOSR grants FA9550-24-1-0254 and DOE grant DE-SC0023164. Andrew Christlieb and Sining Gong are also supportted by ONR grant N00014-24-1-2242. Jing-Mei Qiu and Nanyi Zheng are also supported by AFOSR FA9550-22-1-0390. The authors also thank Michael S. Murillo from Computational Mathematics, Science, and Engineering at Michigan State University for valuable conversations regarding the Wigner–Poisson system. The authors acknowledge the help with grammar from the ChatGPT. 

\appendix

\section{Connection Between the Wigner-Poisson and Vlasov-Poisson Systems}\label{appendixA}
In this appendix, we present a careful derivation that demonstrates how the Wigner-Poisson system reduces to the Vlasov-Poisson system in the classical limit as the non-dimensional Planck constant $H \to 0$. While it is well understood for many years that this correspondence holds, inconsistencies and sign errors occasionally appear in the literature due to mistakes in this derivation. We therefore provide a short, self-contained review to clarify this connection.

Let the right-hand side of Equation~\ref{eq:wigner} be denoted as $W(H)$:
\[
W(H) := -\frac{i}{2 \pi H^2} \iint dv' \, dx' \, \exp\left( i \frac{v' - v}{H} x' \right) \left[ \Phi\left(x + \frac{x'}{2}\right) - \Phi\left(x - \frac{x'}{2}\right) \right] f(x,v', t).
\]
To simplify the expression, we perform a change of variables by letting $\tilde{x} = \frac{x'}{H}$, which yields:
\[
W(H) = -\frac{i}{2 \pi H} \iint dv' \, d\tilde{x} \, \exp\left( i (v' - v)\tilde{x} \right) \left[ \Phi\left(x + \frac{H\tilde{x}}{2}\right) - \Phi\left(x - \frac{H\tilde{x}}{2}\right) \right] f(x,v', t).
\]
Taking the limit as $H \to 0$, we obtain:
\[
\lim_{H \to 0} W(H) = \lim_{H \to 0} -\frac{i}{2 \pi} \iint dv' \, d\tilde{x} \, \exp\left( i (v' - v)\tilde{x} \right) \left[ \frac{\Phi\left(x + \frac{H\tilde{x}}{2}\right) - \Phi\left(x - \frac{H\tilde{x}}{2}\right)}{H} \right] f(x,v', t).
\]
We now focus on the inner expression. Note that:
\[
\lim_{H \to 0} \frac{\Phi\left(x + \frac{H \tilde{x}}{2}\right) - \Phi\left(x - \frac{H \tilde{x}}{2}\right)}{H}
= \tilde{x} \cdot \lim_{H \to 0} \frac{\Phi\left(x + \frac{H \tilde{x}}{2}\right) - \Phi\left(x - \frac{H \tilde{x}}{2}\right)}{H \tilde{x}} = \tilde{x} \, \partial_x \Phi(x).
\]
Substituting this into the earlier expression gives:
\[
\lim_{H \to 0} W(H) = -\frac{i\partial_x \Phi(x)}{2\pi} \iint dv' \, d\tilde{x} \, \tilde{x} \, \exp(i(v' - v)\tilde{x}) f(x, v', t).
\]
We now use the identity:
\[
\frac{\partial}{\partial v'} \left( \exp(i(v' - v)\tilde{x}) \right) = i \tilde{x} \exp(i(v' - v)\tilde{x})
\]
Thus, we can rewrite the integral and apply integration by parts:
\begin{align*}
    \lim_{H \to 0} W(H) &= -\frac{\partial_x \Phi(x)}{2\pi} \iint dv' \, d\tilde{x} \frac{\partial}{\partial v'} \left( \exp(i(v' - v)\tilde{x}) \right) f(x, v', t) \\
    &= \frac{\partial_x \Phi(x)}{2\pi} \iint dv' \, d\tilde{x} \, \exp(i(v' - v)\tilde{x}) \cdot \partial_{v'} f(x, v', t).
\end{align*}
Using the identity:
\[
\int d\tilde{x} \, \exp(i(v' - v)\tilde{x}) = 2\pi \delta(v' - v),
\]
we obtain that
\begin{align*}
    \lim_{H \to 0} W(H) = \frac{\partial_x \Phi}{2\pi} \int d{v'} \,  \partial_{v'} f(x, v', t) \cdot 2\pi \delta(v' - v) = \partial_x \Phi \, \partial_v f(x, v, t).
\end{align*}
This confirms that in the classical limit $H \to 0$, the Wigner-Poisson system reduces to the Vlasov-Poisson system:
\begin{align*}
    \frac{\partial f}{\partial t} + v \frac{\partial f}{\partial x} - \partial_x \Phi \cdot \partial_v f &= 0, \\
    -\nabla^2 \Phi &= \int f \, dv - 1.
\end{align*}

\bibliographystyle{siam} 
\bibliography{ref}

\end{document}